\pdfoutput=1
\documentclass[reqno, 12pt]{amsart}

\def\ssign{\textsection\nobreak\hspace{1pt plus 0.3pt}}
\makeatletter
\newif\if@sectsign \@sectsigntrue
\def\@seccntformat#1{\protect\textup{\protect\@secnumfont
		\ifnum\pdfstrcmp{#1}{section}=\z@
			\if@sectsign \protect\ssign\csname the#1\endcsname\protect\enspace
			\else \csname the#1\endcsname\protect\@secnumpunct\fi
		\else
			\csname the#1\endcsname\protect\@secnumpunct
		\fi
	}}
\g@addto@macro\appendix{\@sectsignfalse}
\makeatother

\usepackage{amsmath,amssymb,amsthm}
\usepackage{mathrsfs}
\usepackage{mathabx}\changenotsign
\usepackage{dsfont}
\usepackage{scalerel}
\usepackage{graphicx}

\usepackage[T1]{fontenc}
\usepackage{lmodern}
\usepackage[british]{babel}
\usepackage[babel]{microtype}

\usepackage{geometry}
\usepackage{xcolor}
\usepackage[backref]{hyperref}
\hypersetup{
	colorlinks,
	linkcolor={red!60!black},
	citecolor={green!60!black},
	urlcolor={blue!60!black}
}

\usepackage[open,openlevel=2,atend]{bookmark}

\usepackage[abbrev,msc-links,backrefs]{amsrefs}
\usepackage{doi}

\renewcommand{\PrintDOI}[1]{\doi{#1}}

\numberwithin{equation}{section}
\numberwithin{figure}{section}

\usepackage{calc}\usepackage{enumitem}
\def\rmlabel{\upshape({\itshape \roman*\,})}

\def\alabel{\upshape(\makebox[\widthof{\itshape a}][c]{\itshape \alph*}\,)}
\def\Alabel{\upshape({\itshape \Alph*\,})}

\theoremstyle{plain}
\newtheorem{thm}{Theorem}[section]
\newtheorem{fact}[thm]{Fact}
\newtheorem{prop}[thm]{Proposition}

\newtheorem{lemma}[thm]{Lemma}

\theoremstyle{definition}
\newtheorem{dfn}[thm]{Definition}

\theoremstyle{remark}

\newtheorem{clm}[thm]{Claim}

\let\eps=\varepsilon
\let\theta=\vartheta
\let\rho=\varrho
\let\phi=\varphi

\let\polishlcross=\l
\DeclareRobustCommand{\l}{\ifmmode\ell\else\polishlcross\fi}
\ifdefined{\def\l{l}}\fi
\usepackage{textcase}

\def\NN{\mathds N}

\def\ZZ{\mathds Z}

\makeatletter
\def\@defcal#1{\expandafter\def\csname c#1\endcsname{{\mathcal{#1}}}}
\def\@defscr#1{\expandafter\def\csname cc#1\endcsname{{\mathscr{#1}}}}
\def\@deffrak#1{\expandafter\def\csname f#1\endcsname{{\mathfrak{#1}}}}
\@tfor\next:=ABCDEFGHIJKLMNOPQRSTUVWXYZ\do{\expandafter\@defcal\next}
\@tfor\next:=ABCDEFGHIJKLMNOPQRSTUVWXYZ\do{\expandafter\@defscr\next}
\@tfor\next:=ABCDEFGHIJKLMNOPQRSTUVWXYZ\do{\expandafter\@deffrak\next}
\@tfor\next:=abcdefghjklmnopqrstuvwxyz\do{\expandafter\@deffrak\next}\makeatother

\makeatletter
\def\moverlay{\mathpalette\mov@rlay}
\def\mov@rlay#1#2{\leavevmode\vtop{ \baselineskip\z@skip
		\lineskiplimit-\maxdimen
\ialign{\hfil$\m@th#1##$\hfil\cr#2\crcr}}}
\newcommand{\charfusion}[3][\mathord]{
	#1{\ifx#1\mathop\vphantom{#2}\fi
		\mathpalette\mov@rlay{#2\cr#3}
	}
\ifx#1\mathop\expandafter\displaylimits\fi}
\makeatother

\newcommand{\dcup}{\charfusion[\mathbin]{\cup}{\cdot}}
\newcommand{\bigdcup}{\charfusion[\mathop]{\bigcup}{\cdot}}

\DeclareFontFamily{U} {MnSymbolC}{}
\DeclareSymbolFont{MnSyC} {U} {MnSymbolC}{m}{n}
\DeclareFontShape{U}{MnSymbolC}{m}{n}{
	<-6> MnSymbolC5
	<6-7> MnSymbolC6
	<7-8> MnSymbolC7
	<8-9> MnSymbolC8
	<9-10> MnSymbolC9
	<10-12> MnSymbolC10
<12-> MnSymbolC12}{}
\DeclareMathSymbol{\powerset}{\mathord}{MnSyC}{180}

\def\tand{\ \text{and}\ }
\def\qand{\quad\text{and}\quad}
\def\qqand{\qquad\text{and}\qquad}

\let\vn=\varnothing

\let\lra=\longrightarrow
\let\to=\lra

\makeatletter
\def\@shortto@sub_#1{_{\begingroup\let\to\rightarrow#1\endgroup}}
\newcommand{\@withshortto}[1]{#1\@ifnextchar_\@shortto@sub\relax}
\@for\@op:=lim,limsup,liminf,varlimsup,varliminf,varinjlim,varprojlim,projlim,injlim\do{\@ifundefined{\@op}{}{\expandafter\let\csname @save@\@op\expandafter\endcsname\csname\@op\endcsname
		\expandafter\edef\csname\@op\endcsname{\noexpand\@withshortto\expandafter\noexpand\csname @save@\@op\endcsname}}}
\makeatother

\usepackage{tikz}
\usetikzlibrary{calc}

\newcommand{\tedgefull}[5][]{\path[line join=round, #1]
	let \p1=($#4-#3$), \p2=($#5-#4$), \p3=($#3-#5$),
	    \n1={atan2(\y1,\x1)}, \n2={atan2(\y2,\x2)}, \n3={atan2(\y3,\x3)},
	    \n9={Mod(\n2-\n1,360)<180 ? 1 : -1}in
	($#3+(\n1-\n9*90:#2)$) -- ($#4+(\n1-\n9*90:#2)$)
	arc[start angle=\n1-\n9*90, delta angle={\n9*Mod(\n9*(\n2-\n1),360)}, radius=#2]
	-- ($#5+(\n2-\n9*90:#2)$)
	arc[start angle=\n2-\n9*90, delta angle={\n9*Mod(\n9*(\n3-\n2),360)}, radius=#2]
	-- ($#3+(\n3-\n9*90:#2)$)
	arc[start angle=\n3-\n9*90, delta angle={\n9*Mod(\n9*(\n1-\n3),360)}, radius=#2]
	-- cycle;
}

\newdimen\hvertexradius
\newcommand{\hedgeoffset}{2.25\hvertexradius}
\tikzset{
	hedge/.style={line width=1.25pt, fill opacity=0.2},
	tedge/.style={hedge, draw=red!70!black, fill=red!70!black},
}

\newcommand{\tedge}[4][]{\tedgefull[tedge, #1]{\hedgeoffset}{#2}{#3}{#4}}

\pgfdeclarelayer{foreground}
\pgfsetlayers{main,foreground}
\newcommand{\hvertex}[1]{\begin{pgfonlayer}{foreground}
		\fill #1 circle [radius=\hvertexradius];
	\end{pgfonlayer}}

\theoremstyle{plain}

\theoremstyle{definition}
\newtheorem{setup}[thm]{Setup}

\theoremstyle{remark}

\def\bl{\bigl(}
\def\br{\bigr)}

\makeatletter
\newcommand{\pushright}[1]{\ifmeasuring@#1\else\omit\hfill$\displaystyle#1$\fi\ignorespaces}
\newcommand{\pushleft}[1]{\ifmeasuring@#1\else\omit$\displaystyle#1$\hfill\fi\ignorespaces}
\makeatother

\let\sm=\smallsetminus

\def\csr{c_{\rm SR}}
\def\cdr{c_{\rm DR}}

\def\Cs{\cC_7^{-}}
\def\fQcr{\fQ_{\mathrm{cr}}}

\DeclareMathOperator{\Ind}{Ind}
\DeclareMathOperator{\Triad}{Triad_{\ee}}
\DeclareMathOperator{\reg}{reg}

\DeclareMathOperator{\SZRL}{SzRL}

\newcommand{\g}[1]{\widetilde{#1}}

\makeatletter

\newcommand{\glyphscale}{1.2}
\newcommand{\glyphkern}{0.75mu}
\newcommand{\glyphdotr}{.35}
\newcommand{\glyphringr}{.30}
\newcommand{\glyphringw}{0.10cm}
\newcommand{\glyphedgew}{0.28cm}
\newcommand{\glyphcompanions}{{pi}{\pi},{d}{d},{e}{e},{ccK}{\ccK},{E}{E}}

\@namedef{glyph@skip@eee}{}

\newcommand{\gdot}[1]{\draw[fill] (#1:1) circle (\glyphdotr);}
\newcommand{\gcirc}[1]{\draw[line width=\glyphringw] (#1:1) circle (\glyphringr);}
\newcommand{\gedge}[2]{\draw[line width=\glyphedgew] (#1:1) -- (#2:1);}

\newcommand{\glyph@use}[2]{\edef\glyph@now{\current@color}\expandafter\ifx\csname glyph@#1@col\endcsname\glyph@now\else
		\expandafter\global\expandafter\setbox\csname glyph@#1@box\endcsname
			\hbox{\tikz{#2}}\expandafter\xdef\csname glyph@#1@col\endcsname{\glyph@now}\fi
	\mathord{\mkern\glyphkern
		\scaleobj{\glyphscale}{\scalerel*{\usebox{\csname glyph@#1@box\endcsname}}{x}}\mkern\glyphkern}}

\newcommand{\glyph@companion}[3]{\ifcsname glyph@skip@#1#3\endcsname\else
		\ifcsname #1#3\endcsname
			\errmessage{NewGlyph: companion
				\expandafter\string\csname #1#3\endcsname\space is already defined}\fi
		\expandafter\DeclareRobustCommand\csname #1#3\endcsname{#2_{\csname #3\endcsname}}\fi
}
\newcommand{\NewGlyph}[2]{\expandafter\newsavebox\csname glyph@#1@box\endcsname
	\expandafter\gdef\csname glyph@#1@col\endcsname{\glyph@unset}\expandafter\DeclareRobustCommand\csname #1\endcsname{\glyph@use{#1}{#2}}\@for\glyph@pair:=\glyphcompanions\do{\expandafter\glyph@companion\glyph@pair{#1}}}
\makeatother

\NewGlyph{vdeg}{\gdot{90}\gcirc{210}\gcirc{330}}
\NewGlyph{vv}  {\gcirc{90}\gdot{210}\gdot{330}}
\NewGlyph{pdeg}{\gcirc{90}\gdot{210}\gdot{330}\gedge{210}{330}}
\NewGlyph{vvv} {\gdot{90}\gdot{210}\gdot{330}}
\NewGlyph{ev}  {\gdot{90}\gdot{210}\gdot{330}\gedge{210}{330}}
\NewGlyph{ee}  {\gdot{90}\gdot{210}\gdot{330}\gedge{90}{330}\gedge{90}{210}}
\NewGlyph{eee} {\gdot{90}\gdot{210}\gdot{330}\gedge{90}{330}\gedge{90}{210}\gedge{210}{330}}

\begin{document}
\title[Regularity method for hypergraphs with~$4$-cycle-free links]{Regularity method for hypergraphs with~$4$-cycle-free links}

\author[A.~Basu]{Ayush Basu}
\address{Department of Mathematics, Emory University, Atlanta, USA}
\email{ayush.basu@emory.edu}

\author[Chr.~Reiher]{Christian Reiher}
\address{Fachbereich Mathematik, Universit\"at Hamburg, Hamburg, Germany}
\email{christian.reiher@uni-hamburg.de}

\author[V.~R\"{o}dl]{Vojt\v{e}ch R\"{o}dl}
\address{Department of Mathematics, Emory University, Atlanta, USA}
\email{vrodl@emory.edu}

\author[M.~Schacht]{Mathias Schacht}
\address{Fachbereich Mathematik, Universit\"at Hamburg, Hamburg, Germany}
\email{schacht@math.uni-hamburg.de}

\thanks{The third author is supported by NSF grant DMS~2300347.}


\begin{abstract}
	We extend the hypergraph regularity method to sparse~$3$-uniform
	hypergraphs whose vertex links are~$C_4$-free. In other words, we consider hypergraphs
	$H=(V,E)$ that are~$K_{1,2,2}$-free, which implies that $|E|=O(|V|^{5/2})$.
	For such hypergraphs we establish a sparse analogue of the removal lemma
	for the tight cycle on seven vertices minus an edge.
\end{abstract}

\maketitle

\section{Introduction}
\label{sec:introduction}
The following hypergraph removal lemma, extending the Ruzsa--Szemer\'edi theorem~\cite{RS78} to~$k$-uniform hypergraphs for $k\geq 3$, was proved in~\cites{G07, NRS06, RS04}. Roughly speaking, it asserts that a~$k$-uniform hypergraph on~$n$ vertices containing only $o(n^{v(F)})$ copies of~$F$ can be made~$F$-free by removing $o(n^k)$ hyperedges.
\begin{thm}[Hypergraph removal lemma]
	\label{thm:graphremovallemma}
	For every integer $k\geq 2$, every~$k$-uniform hypergraph~$F$, and every $\eps>0$, there exist some $\delta>0$ and a positive integer~$n_0$
	such that the following holds: If~$H$ is a~$k$-uniform hypergraph on $n\geq n_0$ vertices
	with fewer than~$\delta n^{v(F)}$ copies of~$F$,
	then~$H$ can be made~$F$-free by removing fewer than $\eps n^k$ hyperedges. \qed
\end{thm}
This theorem has had several applications to dense~$k$-uniform hypergraphs on~$n$ vertices, i.e., those with $\Omega(n^k)$ hyperedges. It is natural to consider a sparse analogue, that is, a result for hypergraphs~$H$ with $pn^k$ hyperedges where $p\to 0$ as $n \to \infty$. A natural `sparse generalisation' of the removal lemma would assert that a hypergraph~$H$ with $o(p^{e(F)}n^{v(F)})$ copies of~$F$ can be made~$F$-free by removing $o(pn^k)$ hyperedges. However, such a statement is false in general, and finding the right conditions under which a `sparse removal lemma' does hold is a delicate question.

In recent years, progress has been made in this direction for graphs that are subgraphs of a random graph. Extending the approach of~\cite{RS78} to prove a removal lemma in this setting would require a sparse analogue of the regularity lemma~\cite{Sz78} and an associated counting lemma. The regularity lemma for subgraphs of sparse random graphs was obtained in~\cites{KRSparse, scott}
and a corresponding counting lemma appeared in~\cites{CGSS14, BMS15, ST15}.

In another direction, a removal lemma for sparse deterministic graphs that contain `few' $4$-cycles was proved by Conlon, Fox, Sudakov,
and Zhao~\cite{CFSZ21}. Recall that~$C_4$-free graphs on~$n$ vertices have at most $(1/2+o(1))n^{3/2}$ edges
(see~\cite{KST54}), i.e.,~$O(pn^{2})$ edges where $p=n^{-1/2}$. Further, one easily checks that~$C_4$-free graphs contain~$O(n^{5/2})$ copies of~$C_5$. The corresponding removal lemma from~\cite{CFSZ21} establishes
that such~$C_4$-free graphs containing~$o(p^5n^5)$ copies of~$C_5$ can be made~$C_5$-free by removing~$o(pn^2)$ edges.
\begin{thm}[Conlon, Fox, Sudakov \& Zhao]
	\label{thm:C4Spremoval}
	For every $\mu > 0$ there exist some $c>0$ and a positive integer~$n_0$ such that the following holds: If~$G$ is a~$C_4$-free graph on $n\geq n_0$ vertices containing fewer than $c n^{5/2}$ copies of~$C_5$, then~$G$ can be made~$C_5$-free by removing fewer than $\mu n^{3/2}$ edges. \qed
\end{thm}
We prove an analogous statement for a class of sparse~$3$-uniform hypergraphs. To state the result, we first introduce the following~$3$-uniform hypergraphs which play the r\^ole of~$C_5$ and~$C_4$ in our analogue of Theorem~\ref{thm:C4Spremoval}.

Given a positive integer $k\geq 4$, let the tight cycle~$\cC_k$ on~$k$ vertices be the~$3$-uniform hypergraph with vertex set~$\ZZ/k\ZZ$ and edge set
\[
	E(\cC_k)= \big\{\{i-1, i, i+1\}\colon i\in \ZZ/k\ZZ\big\}\,.
\]
Similarly, for $k\geq 4$, let~$\cC_k^{-}$, the tight cycle on~$k$ vertices minus an edge, be the~$3$-uniform hypergraph with vertex set~$\ZZ/k\ZZ$ and
\[
	E(\cC_k^{-})= \big\{\{i-1, i, i+1\}\colon i\in (\ZZ/k\ZZ)\sm\{0\}\big\}\,.
\]
The hypergraph~$\Cs$, which plays the r\^ole of~$C_5$, is illustrated in Figure~\ref{fig:c7minus}.
\begin{figure}[ht]
	\centering
	\begin{tikzpicture}
		\foreach \k in {0,...,6}
		\coordinate (v\k) at ({90-360/7*\k}:2.1);
		\foreach \a/\b/\c in {0/1/2, 1/2/3, 2/3/4, 3/4/5, 4/5/6, 5/6/0}
		{\tedge{(v\a)}{(v\b)}{(v\c)}}
		\foreach \k in {0,...,6}
		{\hvertex{(v\k)}}
		\foreach \k in {0,...,6}
		\node at ({90-360/7*\k}:2.5) {$\k$};
	\end{tikzpicture}
	\caption{The hypergraph~$\Cs$: the six shaded triples are its hyperedges, and the triple $\{6,0,1\}$
		is the removed hyperedge from~$\cC_7$.}
	\label{fig:c7minus}
\end{figure}
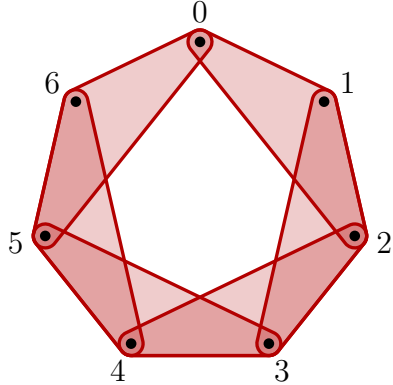

We denote by~$K_{1,2,2}$ the complete tripartite~$3$-uniform hypergraph with four hyperedges and vertex classes of size one, two, and two. This hypergraph will play the r\^ole of~$C_4$. Recall that the \emph{link} of a vertex~$w$ in a~$3$-uniform hypergraph~$H$ is the graph with vertex set $V(H)\sm\{w\}$ and edge set $\{xy\colon wxy\in E(H)\}$. Observe that~$H$ is~$K_{1,2,2}$-free if and only if no link of a vertex of~$H$ contains a~$C_4$. Just as for~$C_4$-free graphs, the extremal edge-density for~$K_{1,2,2}$-free~$3$-uniform hypergraphs is~$\Theta(n^{-1/2})$. Indeed, since the link of each vertex in a~$K_{1,2,2}$-free hypergraph~$H$ is~$C_4$-free, $H$ can have at most
$
	\frac{n}{3}\cdot \big(\frac{1}{2}+o(1)\big)n^{3/2}
	=
	\big(\frac{1}{6}+o(1)\big)n^{5/2}
$
hyperedges. On the other hand, it was shown by Mubayi~\cite{Mub02} that there exist~$K_{1,2,2}$-free hypergraphs on~$n$ vertices with~$\Omega(n^{5/2})$ hyperedges. Furthermore, it can be shown that the~$K_{1,2,2}$-free hypergraph constructed in~\cite{Mub02} contains $\Theta(n^4)$ copies of~$\Cs$.

Our main result parallels Theorem~\ref{thm:C4Spremoval} in the context of hypergraphs. We show that a~$K_{1,2,2}$-free hypergraph on~$n$ vertices with $o(n^4)$ copies of~$\Cs$ can be made~$\Cs$-free by removing $o(n^{5/2})$ hyperedges.
In the following statement, the subscript~$\mathrm{SR}$ of the constant~$\csr$ stands for `sparse removal'.
\begin{thm}[Sparse removal lemma for~$\Cs$]
	\label{thm:spremovalmain}
	For every $\mu>0$, there exists $\csr>0$ such that for sufficiently
	large~$n$ and $p = n^{-1/2}$ the following holds.

	If a~$K_{1,2,2}$-free~$3$-uniform hypergraph~$H$ on~$n$ vertices contains
	at most $\csr p^6n^7$ copies of~$\Cs$, then it can be
	made~$\Cs$-free by removing fewer than $\mu pn^3$ hyperedges.
\end{thm}
We remark that every~$K_{1,2,2}$-free hypergraph~$H$ on~$n$ vertices contains at most~$n^4$ copies of~$\Cs$, so the hypothesis of Theorem~\ref{thm:spremovalmain} lies only a constant factor below this trivial bound. Indeed, as Figure~\ref{fig:K122free} shows, for fixed vertices~$w$,~$y$, and~$z$ there is at most one vertex~$x$ such that both triples $wxy$ and~$xyz$ are hyperedges in~$H$, since two such vertices would span a copy of~$K_{1,2,2}$. Here and below, we call a quadruple $wxyz$ of vertices a \emph{tight path} in~$H$ if both $wxy$ and~$xyz$ are hyperedges of~$H$. Now choose the vertices $v_0,v_2,v_3,v_5$ of a copy of~$\Cs$. Once these four vertices are fixed, the same observation, applied successively to the tight paths $v_0v_1v_2v_3$, $v_5v_4v_3v_2$, and $v_0v_6v_5v_4$, determines at most one choice for each of~$v_1$, $v_4$, and~$v_6$. Thus~$H$ contains at most~$n^4$ copies of~$\Cs$.

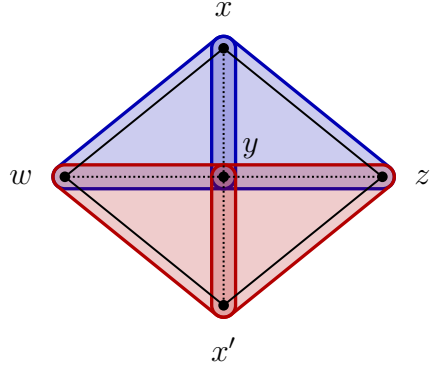
\begin{figure}[ht]
	\centering
	\begin{tikzpicture}[line width=0.75pt]
		\coordinate (w)  at (-2.1,0);
		\coordinate (x)  at (0,1.7);
		\coordinate (z)  at (2.1,0);
		\coordinate (xp) at (0,-1.7);
		\coordinate (y)  at (0,0);
		\tedge[draw=blue!70!black, fill=blue!70!black]{(y)}{(w)}{(x)}
		\tedge[draw=blue!70!black, fill=blue!70!black]{(y)}{(x)}{(z)}
		\tedge{(y)}{(z)}{(xp)}
		\tedge{(y)}{(xp)}{(w)}
		\draw (w) -- (x) -- (z) -- (xp) -- cycle;
		\draw[densely dotted] (y) -- (w) (y) -- (x) (y) -- (z) (y) -- (xp);
		\foreach \p in {w,x,xp,y,z}
		{\hvertex{(\p)}}
		\node[left=8pt]        at (w)  {$w$};
		\node[above=8pt]       at (x)  {$x$};
		\node[right=8pt]       at (z)  {$z$};
		\node[below=8pt]       at (xp) {$x'$};
		\node[above right=3pt] at (y)  {$y$};
	\end{tikzpicture}
	\caption{The vertices $\{w,x,x',y,z\}$ form a~$K_{1,2,2}$ with vertex classes~$\{y\}$, $\{x,x'\}$, and $\{w,z\}$: the four shaded regions are its four hyperedges, and the link of~$y$ is the~$4$-cycle $wxzx'$ (solid). The upper blue (resp.\ lower red) hyperedges form the tight path $wxyz$ (resp.\ $wx'yz$).}
	\label{fig:K122free}
\end{figure}

The rest of the paper is organised as follows. In~\ssign\ref{sec:outline} we
introduce the sparse hypergraph regularity method for~$K_{1,2,2}$-free
hypergraphs and deduce Theorem~\ref{thm:spremovalmain} from four propositions.
The proofs of those propositions are given in \S\ssign\ref{sec:crowded}\,--\,\ref{app:regularity}.

\section{Outline of the proof}\label{sec:outline}
The proof has three ingredients: a sparse hypergraph regularity lemma together with a
structural consequence of $K_{1,2,2}$-freeness (see Propositions~\ref{prop:regularity} and~\ref{prop:crowded} below),
a sparse counting lemma for~$\Cs$ (see Proposition~\ref{prop:counting}), and a
removal lemma for homomorphic images of~$\Cs$ in reduced hypergraphs arising in the context of the hypergraph regularity lemma (see Proposition~\ref{prop:reducedremoval}).
We introduce the relevant notation and then deduce Theorem~\ref{thm:spremovalmain} from those propositions.

The first ingredient is a
customised version of the sparse hypergraph regularity lemma,
which can be proved in a standard way.
For two disjoint sets~$X$ and~$Y$ we denote by~$K[X,Y]$ the complete bipartite graph
with these vertex classes.
We say that a bipartite graph $P=(X\dcup Y,E)$ is \emph{$(\eps, d)$-regular} if for all subsets
$X'\subseteq X$ and $Y'\subseteq Y$ we have
\[
	\big|e(X',Y')-d|X'||Y'|\big|\le \eps |X||Y|\,,
\]
where $e(X',Y')$ denotes the number of edges of~$P$ with one vertex in~$X'$ and one
vertex in~$Y'$. 
We say~$P$ is
\emph{$\eps$-regular} if it is $(\eps,d)$-regular for some $d\in[0,1]$.

By a~$\ee$-partite graph we mean a tripartite graph $P=(X\dcup Y\dcup Z,E)$
without any edges from~$X$ to~$Z$. Formally, this is the same as a bipartite graph
with vertex classes~$Y$ and $X\dcup Z$ together with a fixed partition of the second
class into~$X$ and~$Z$, but we think of~$\ee$-partite graphs as tripartite rather
than bipartite.
Such a~$\ee$-partite graph $P=(X\dcup Y\dcup Z, E)$ is called \emph{$(\eps, d)$-regular} (resp.\ \emph{$\eps$-regular})
if its induced bipartite subgraphs $P[X, Y]$ and~$P[Y, Z]$ are $(\eps, d)$-regular (resp.\ $\eps$-regular).

\begin{dfn}[Equitable partition]\label{def:equit-partition}
	For~$\tau>0$ and positive integers~$t$, $\l$
	we say that a family of bipartite graphs
	\[
		\ccP=\big\{P^{ij}_\alpha\colon ij\in[t]^{(2)}\tand \alpha\in[\l]\big\}
	\]
	is a \emph{$(\tau,t,\l)$-partition} on a given vertex set~$V$
	if there exists a partition $V_0\dcup V_1\dcup\cdots\dcup V_t=V$ such that
	\begin{enumerate}[label=\rmlabel]
		\item $|V_0|\leq \tau|V|$ and $|V_1|=\dots=|V_t|$, and
		\item for all $ij\in[t]^{(2)}$, the bipartite graph $K[V_i,V_j]$
		      is partitioned as
		      \[
			      K[V_i,V_j] = \bigdcup_{\alpha\in [\l]}P^{ij}_\alpha\,.
		      \]
	\end{enumerate}
	If, in addition, $\eps>0$ and the bipartite graphs~$P^{ij}_\alpha$ are
	$(\eps,1/\l)$-regular for all $ij\in[t]^{(2)}$ and $\alpha\in [\l]$,
	we say that~$\ccP$
	is a \emph{$(\tau,t,\l,\eps)$-equitable partition}.
\end{dfn}

Given~$\ccP$ as above, let $i,j,k\in[t]$ be pairwise distinct and let
$\alpha,\beta\in[\l]$. The ordered pair
\[
	(P^{ij}_\alpha,P^{jk}_\beta)
\]
will be called a~$\ee$-triad of~$\ccP$. We write
\[
	P^{ijk}_{\alpha\beta}=P^{ij}_\alpha\dcup P^{jk}_\beta
\]
for the associated~$\ee$-partite graph with vertex partition
$V_i\dcup V_j\dcup V_k$.
We shall
freely identify a~$\ee$-triad $(P^{ij}_\alpha,P^{jk}_\beta)$ with its associated graph $P^{ijk}_{\alpha\beta}$
whenever the meaning is clear.
The set of all~$\ee$-triads is denoted by
\[
	\Triad(\ccP)
	=
	\bigl\{(P^{ij}_\alpha,P^{jk}_\beta)\simeq P^{ijk}_{\alpha\beta}
	\colon i,j,k\in[t]\ \text{pairwise distinct and}\
	\alpha,\beta\in[\l]\bigr\}\,.
\]
Here and below, the ordering of $i,j,k$ in the superscript is part of
the notation. Thus, for example, $P^{ijk}_{\alpha\beta}$ and~$P^{ikj}_{\alpha\beta}$ in general denote different graphs.

Having described the type of partition the sparse regularity lemma provides, we
need to discuss how a given hypergraph~$H$ can interact with such~$\ee$-triads. In
general, for a~$\ee$-partite graph $P=(X\dcup Y\dcup Z, E_P)$ we set
\[
	\ccKee(P)=\{(x, y, z)\in X\times Y\times Z\colon xy, yz\in E_P\}\,.
\]
If in addition to~$P$ we have a hypergraph $H=(V, E)$
with $V\supseteq X\dcup Y\dcup Z$, we set
\[
	\Eee(H\,|\,P)=\bigl\{xyz\in E\colon (x, y, z)\in \ccKee(P)\bigr\}
\]
as well as $e_{\ee}(H\,|\,P)=\bigl|\Eee(H\,|\,P)\bigr|$ and
\[
	\dee(H\,|\,P)=\frac{e_{\ee}(H\,|\,P)}{|\ccKee(P)|}\,,
\]
with the convention that $\dee(H\,|\,P)=0$ when $|\ccKee(P)|=0$ (which, however, never
occurs below).
In analogy with the notion of sparse regularity for graphs
in~\cites{KRSparse,scott}, we define a concept of hypergraph regularity
suitable for sparse~$3$-uniform hypergraphs and~$\ee$-triads.

\begin{dfn}[Hypergraph regularity]\label{def:reg}
	A~$3$-uniform hypergraph $H=(V,E)$ is \emph{$(\delta,d,p)$-regular with respect to
		a~$\ee$-partite graph $P=(X\dcup Y\dcup Z,E_P)$}
	with $V\supseteq X\dcup Y\dcup Z$ if for every~$\ee$-partite subgraph $Q\subseteq P$
	with $|\ccKee(Q)|\ge \delta|\ccKee(P)|$ we have
	\[
		\big|\dee(H\,|\,Q)-d\big|\le \delta p\,.
	\]
	Moreover, we simply say \emph{$H$ is $(\delta,p)$-regular with respect to~$P$}
	if it is $(\delta,d,p)$-regular for $d=\dee(H\,|\,P)$.
\end{dfn}

With these definitions at hand, we can state a sparse version of the hypergraph regularity
lemma suitable for our needs, a proof of which is sketched in~\ssign\ref{app:regularity}.

\begin{prop}[Sparse hypergraph regularity lemma]
	\label{prop:regularity}
	For all $\tau, \delta > 0$, $\eps\colon\NN\to (0,1]$, and positive integers~$t_0$,
	$\l_0$, there exist positive integers~$T_0$, $L_0$, and~$n_0$ such that for every
	integer $n\geq n_0$ and $p=n^{-1/2}$ the following holds.

	Every~$K_{1,2,2}$-free~$3$-uniform hypergraph~$H$ on~$n$ vertices with at most $pn^3$ hyperedges
	admits for some integers~$t$ and~$\l$ with $t_0\leq t\leq T_0$ and $\l_0\leq \l\leq L_0$
	a $(\tau,t,\l,\eps(\l))$-equitable
	partition~$\ccP$ such
	that for all but at most $\delta \l^2t^3$ of the~$\ee$-triads $P^{ijk}_{\alpha\beta}\in\Triad(\ccP)$,
	the hypergraph~$H$ is $(\delta, p)$-regular
	with respect to $P^{ijk}_{\alpha\beta}$.
\end{prop}

The absence of subhypergraphs isomorphic to~$K_{1,2,2}$ limits how many hyperedges
can accumulate on~$\ee$-triads of unusually high density, whether or not these
triads are regular. To state this precisely, we first fix some terminology.
Arbitrary sets of quintuples $(i, j, k, \alpha, \beta)\in [t]^3\times [\l]^2$
with pairwise distinct~$i$, $j$, $k$ will be called
\emph{$(t, \l)$-quintuple systems}. Moreover, given
an equitable partition~$\ccP$ of a hypergraph~$H$ as well as $C\ge 1$
and $p=n^{-1/2}$, the \emph{crowded}~$\ee$-triads of~$\ccP$ are those indexed
by the quintuple system
\[
	\fQcr(C, p)=\big\{(i, j, k, \alpha, \beta)\in [t]^3\times [\l]^2\colon
	i,j,k\text{ are pairwise distinct and }
	\dee(H\,|\,P^{ijk}_{\alpha\beta})\ge Cp\big\}\,.
\]
The following proposition, which we prove in~\ssign\ref{sec:crowded}, bounds
the number of hyperedges supported by crowded~$\ee$-triads.

\begin{prop}[Crowded triads]\label{prop:crowded}
	Let $C\ge 1$, $t, \l\in \NN$, $\eps\le (8\l^2)^{-1}$, $\tau\le \frac 14$, and $p=n^{-1/2}$.
	If~$\ccP$ denotes a~$(\tau,t,\l,\eps)$-equitable partition of a~$K_{1,2,2}$-free
	hypergraph~$H$ on $n\ge (32\l t)^2$ vertices and
	\[
		m
		=|V_1|
		=\dots
		=|V_t|\,,
	\]
	then
	\[
		\sum\big\{e_{\ee}(H\,|\,P^{ijk}_{\alpha\beta})\colon (i, j, k, \alpha, \beta)\in \fQcr(C, p)\big\} \le \frac{2}{C}pn^3\,.
	\]
\end{prop}

Proposition~\ref{prop:crowded} therefore suggests discarding the crowded~$\ee$-triads
and retaining only those that are regular and of density at least~$dp$. This
motivates the following concept.

\begin{dfn}[Reduced quintuple system]\label{def:reduced}
	Given a $(\tau,t,\l,\eps)$-equitable partition~$\ccP$ of the vertex set~$V$ of
	a hypergraph~$H=(V,E)$ and positive real numbers~$\delta$, $d$, and~$p$, the \emph{reduced quintuple system}~$\fQ(\delta, d, p)$ of~$H$ with respect to~$\ccP$ is defined by
	\begin{multline*}
		\fQ(\delta, d, p)=\Bigl\{(i, j, k, \alpha, \beta)\in [t]^3\times [\l]^2
		  \colon i,j,k\text{ are pairwise distinct,}          \\
		  \text{$\dee(H\,|\,P^{ijk}_{\alpha\beta})\ge dp$,
		  	and~$H$ is $(\delta, p)$-regular
			with respect to~$P^{ijk}_{\alpha\beta}$}\Bigr\}\,.
	\end{multline*}
\end{dfn}

Suppose now that $v_0v_1v_2v_3v_4v_5v_6$
denotes the vertices of a copy of~$\Cs$ in a hypergraph~$H$ with hyperedges $v_{k-1}v_kv_{k+1}$ for
all non-zero~$k$.
If, in addition, we have a $(\tau,t,\l,\eps)$-equitable partition~$\ccP$ of~$H$,
then by tracing the indices of the vertex classes containing the vertices~$v_k$ and
the indices of the bipartite graphs containing the pairs~$v_kv_{k+1}$ we get fourteen
indices forming the following kind of configuration.

\begin{dfn}[$\Cs$-conspiracy]\label{def:conspiracy}
	A \emph{$\Cs$-conspiracy} in a $(t, \l)$-quintuple system~$\fQ$ is a $14$-tuple
	\[
		(i_0, i_1, i_2, i_3, i_4, i_5, i_6, j_{01}, j_{12}, j_{23}, j_{34}, j_{45}, j_{56}, j_{60})\in [t]^7\times [\l]^7
	\]
	such that, reading indices modulo~$7$,
	we have $(i_{k-1}, i_k, i_{k+1}, j_{k-1, k}, j_{k, k+1})\in \fQ$
	for all non-zero $k\in \ZZ/7\ZZ$.
\end{dfn}

Note that the indices $i_0, \dots, i_6$ in a $\Cs$-conspiracy need not be
distinct. For instance, $i_0=i_3$ is possible.
The connection between copies of~$\Cs$ and~$\Cs$-conspiracies just described will
be elaborated on in~\ssign\ref{sec:counting}, where we establish the following
sparse counting lemma.

\begin{prop}[Sparse counting lemma for~$\Cs$]\label{prop:counting}
	Suppose reals $d, \eta, \tau, \delta,\eps, \lambda\in (0, 1]$, and integers~$t$, $\l\ge 1$
	satisfy
	\[
		\tau\le \frac{1}{65}\,,\qquad
		\delta\le \frac{d^6\eta^3}{2^{27}}\,, \qquad
		\eps\le \frac{d^4\eta^2}{2^{21}\l^{2}}\,,\qquad
		\lambda=\frac{d^{10}\eta^4}{2^{31}\l^3 t^5}\,, \qqand
		t \ge \frac{16}{\eta}\,.
	\]
	If a~$K_{1,2,2}$-free hypergraph~$H$ on~$n$ vertices admits
	a $(\tau,t,\l,\eps)$-equitable partition~$\ccP$ such that
	for some $p\ge n^{-1/2}$ the reduced quintuple system $\fQ(\delta, d, p)$ contains
	at least $\eta \l^7 t^7$ different~$\Cs$-conspiracies, then~$H$ contains at least $\lambda n^4$
	copies of~$\Cs$.
\end{prop}

Finally, an argument based on the dense hypergraph removal lemma, which we shall give
in~\ssign\ref{sec:reduced-removal}, shows that if there are few~$\Cs$-conspiracies in $\fQ(\delta, d, p)$,
then one does not have to delete many quintuples to destroy all of them.
In the following statement, the subscript~$\mathrm{DR}$ of the constant~$\cdr$
stands for `dense removal'.

\begin{prop}[$\Cs$-conspiracy removal lemma]
	\label{prop:reducedremoval}
	For every $\rho>0$ there exists $\cdr>0$ such that for
	all natural numbers~$t$ and~$\l$ the following holds.

	If a $(t, \l)$-quintuple system~$\fQ$ (as introduced before
	Proposition~\ref{prop:crowded}) contains at most $\cdr \l^{7}t^7$ different
	$\Cs$-conspiracies, then there is some~${\fQ'\subseteq \fQ}$ of
	size $|\fQ'|\le \rho\l^{2}t^3$ such that there are no~$\Cs$-conspiracies
	in $\fQ\sm \fQ'$.
\end{prop}

We conclude this section by deducing the main result from these propositions.

\begin{proof}[Proof of Theorem~\ref{thm:spremovalmain} assuming
		Propositions~\ref{prop:regularity},~\ref{prop:crowded},~\ref{prop:counting}, and~\ref{prop:reducedremoval}]

	It suffices to consider $\mu\in(0,1)$.
		Given such a~$\mu$, we need to specify an appropriate constant $\csr>0$.
	We begin by setting
	\[
		C=\frac 8\mu\,, \qquad d= \frac{\mu}{8}\,, \qqand \rho=\frac{\mu}{8C}\,.
	\]
	Proposition~\ref{prop:reducedremoval} applied to~$\rho$ yields some $\cdr>0$.
	Decreasing~$\cdr$ if necessary, we may assume that $\cdr\le 1$.
	Now we set
	\[
		\tau
		=
		\frac{\mu}{65}\,,\qquad
		\delta
		=
		\frac{d^6\cdr^3}{2^{27}}\,, \qquad
		\l_0
		=1\,,\qqand
		t_0
		=\max\left\{\left\lceil\frac{24}{\mu}\right\rceil, \left\lceil\frac{16}{\cdr}\right\rceil\right\}
	\]
	and define the function $\eps\colon \NN\lra (0, 1]$ by
	setting $\eps(\l)=2^{-21}d^4\cdr^2\l^{-2}$ for every $\l\in \NN$.

	The sparse hypergraph regularity lemma (Proposition~\ref{prop:regularity})
	applied to~$\tau$, $\delta$, $\eps(\cdot)$, $t_0$, and~$\l_0$ yields three integers~$T_0$, $L_0$, and~$n_0$. Without loss of generality we can
	assume $n_0\ge (32\mu^{-1}L_0T_0)^2$.
	We claim that the constant
	\[
		\csr=\frac{d^{10}\cdr^4}{2^{32}L_0^3T_0^5}
	\]
	has the desired property.

	Having defined~$\csr$, we consider a~$K_{1,2,2}$-free~$3$-uniform hypergraph $H=(V,E)$
	on~${n\ge n_0}$ vertices containing at most $\csr p^6n^7$ copies of~$\Cs$ for $p=n^{-1/2}$.
	We need to show that~$H$ can be made~{$\Cs$-free} by removing fewer
	than $\mu pn^{3}$ hyperedges.

	Since $K_{1,2,2}\not\subseteq H$ implies $e(H)\le n^{5/2}=pn^3$, the
	sparse regularity lemma provides for some integers~$t$, $\l$ with $t_0\leq t \le T_0$ and $\l_0\le \l\le L_0$
	a $(\tau,t,\l,\eps(\l))$-equitable partition~$\ccP$
	such that for all but at most $\delta \l^2t^3$ of the
	$\ee$-triads the hypergraph~$H$ is
	$(\delta, p)$-regular.

	We denote the vertex partition of~$\ccP$ by
	$
		V= V_0 \dcup V_1 \dcup V_2\dcup \cdots \dcup V_t
	$,
	where $|V_0|\le \tau n$ and~$|V_1|= \cdots = |V_t| = m$, and the partition
	of the pairs by
	\[
		K[V_i,V_j]= \bigdcup_{\alpha\in [\l]}P^{ij}_\alpha
	\]
	for every $ij\in [t]^{(2)}$.

	First, we delete every hyperedge~$e$ satisfying
	\begin{enumerate}[label=\Alabel]
		\item\label{it:xa} $e\cap V_0\ne \vn$ or
		\item\label{it:xb} $|e\cap V_i|\ge 2$ for some $i\in [t]$,
	\end{enumerate}
	and all hyperedges~$e$ supported by a~$\ee$-triad~$P^{ijk}_{\alpha\beta}\in\Triad(\ccP)$ such that
	\begin{enumerate}[label=\Alabel, resume]
		\item\label{it:xc} $\dee(H\,|\,P^{ijk}_{\alpha\beta})\le dp$ or
		\item\label{it:xd} $\dee(H\,|\,P^{ijk}_{\alpha\beta})\ge Cp$.
	\end{enumerate}
	Having deleted all the above hyperedges, we finally delete all hyperedges supported by a~$\ee$-triad~$P^{ijk}_{\alpha\beta}$ such that
	\begin{enumerate}[label=\Alabel, resume]
		\item\label{it:xe} $H$ is not $(\delta, p)$-regular
		      with respect to $P^{ijk}_{\alpha\beta}$.
	\end{enumerate}
	Since~$H$ is~$K_{1,2,2}$-free, every link of a vertex is~$C_4$-free. In particular,
	the maximum vertex degree is bounded by $n^{3/2}$ and the number of hyperedges removed in item~\ref{it:xa}
	amounts to at most~$|V_0|n^{3/2}\leq\tau n^{5/2}$.

	Moreover, for fixed $i\in[t]$ and $u\in V_i$, the bipartite part of the
	link of~$u$ between $V_i\sm\{u\}$ and $V\sm V_i$ contains
	at most $m\sqrt n+n$ edges by the standard two-path count for~$C_4$-free bipartite graphs. Summing over $u\in V_i$, every
	hyperedge meeting~$V_i$ in exactly two vertices is counted twice.
	Hence there are at most~$\big(m^2\sqrt n+mn\big)/2$ such hyperedges.
	The same two-path count shows that every~$C_4$-free graph
	on at most~$m$ vertices has at most~$m^{3/2}$ edges. Since every hyperedge of
	$H[V_i]$ occurs in the links of each of its three vertices,
	\[
		e\big(H[V_i]\big)
		\le \frac13m\cdot m^{3/2}
		\le \frac13m^2\sqrt n\,.
	\]
	Consequently, item~\ref{it:xb} deletes fewer than
	$t(m^2\sqrt n+mn)$ hyperedges. Together with
	$t\ge t_0\ge 24/\mu$,
	the first three items cause the deletion of at most
	\[
		\tau n^{5/2}
		+
		t\cdot\big(m^2\sqrt n+mn\big)
		+
		dp n^3
		\leq
		\big(\tau+t^{-1}+n^{-1/2}+d\big)\cdot n^{5/2}
		<
		\frac{\mu}{4} n^{5/2}
	\]
	hyperedges.

	Proposition~\ref{prop:crowded} tells us that item~\ref{it:xd} leads to
	the deletion of at most $2n^{5/2}/C=\mu n^{5/2}/4$ additional hyperedges.

	Since, by the dense graph counting lemma, each~$\ee$-triad of~$\ccP$
	supports at most
	\[
		\Big(\frac{1}{\l^2}+2\eps(\l)\Big) m^3
		<
		2\frac{m^3}{\l^2}
	\]
	triples and, by the deletion in item~\ref{it:xd}, the density of each
	remaining~$\ee$-triad is at most~$Cp$, item~\ref{it:xe} concerns at most
	\[
		\delta \l^2 t^3\cdot 2Cp\frac{m^3}{\l^2}
		\le
		2C\delta n^{5/2}
		<
		\frac{\mu}{4}n^{5/2}
	\]
	hyperedges. Summarising the discussion so far, the set~$E'$ of hyperedges deleted at this
	moment satisfies $|E'|\le \frac 34\mu n^{5/2}$.

	For the hypergraph $H'=(V,E\sm E')$ the choice of constants above
	allows for an application of Proposition~\ref{prop:counting} with $\eta=\cdr$. Therefore,
	the reduced quintuple system $\fQ(\delta, d, p)$ for $H'$ contains fewer than $\cdr \l^7 t^7$ different~$\Cs$-conspiracies,
	since otherwise~$H'\subseteq H$ would contain at least
	$2\csr n^4=2\csr p^6n^7$
	copies of~$\Cs$.

	Consequently, Proposition~\ref{prop:reducedremoval} applies and as a result there is a
	subset $\fQ'\subseteq \fQ(\delta, d, p)$
	of at most $\rho \l^2 t^3$ quintuples such that $\fQ(\delta, d, p)\sm \fQ'$ contains no~$\Cs$-conspiracies.
	The set~$\fQ'$ corresponds to a collection of at most $\rho \l^2 t^3$ different~$\ee$-triads in $\Triad(\ccP)$.
	In view of the deletion in item~\ref{it:xd}, these~$\ee$-triads~$P^{ijk}_{\alpha\beta}$
	satisfy $\dee(H\,|\,P^{ijk}_{\alpha\beta})\le Cp$ and support at most
	\[
		\rho \l^2 t^3 \cdot 2Cp\frac{m^3}{\l^2}
		\le
		2C\rho n^{5/2}
		=
		\frac{\mu}{4} n^{5/2}
	\]
	hyperedges of~$H$. So we can delete those hyperedges as well, and in total
	fewer than
	\[
		\tfrac34\mu n^{5/2}+\tfrac{\mu}{4} n^{5/2}
		=
		\mu pn^3
	\]
	hyperedges are deleted.

	It remains to observe that no copies of~$\Cs$ survive all these deletions.
	Indeed, a surviving copy avoids the deletions in items~\ref{it:xa}
	and~\ref{it:xb}, so tracing the indices of its vertices and pairs as described
	before Definition~\ref{def:conspiracy} yields fourteen indices. The deletions
	in items~\ref{it:xc}--\ref{it:xe}, together with the deletion of
	the~$\ee$-triads corresponding to~$\fQ'$, ensure that the six quintuples
	arising in this way belong to~$\fQ(\delta,d,p)\sm\fQ'$; they would therefore
	form a~$\Cs$-conspiracy there, which is impossible.
\end{proof}

\section{Crowded triads}\label{sec:crowded}

In this section we prove Proposition~\ref{prop:crowded}.
For a graph~$P$, we write
$N_P(v)$ for the neighbourhood of a vertex~$v$ in~$P$.
Moreover, for two vertices~$x$, $y\in V$ of a $3$-uniform hypergraph $H=(V,E)$ we
denote by
\[
	N_H(x, y)=\{z\in V \colon xyz\in E\}
\]
their neighbourhood.
Finally, we write
\[
	L(x)=\big(V\sm\{x\},\{yz\colon xyz\in E\}\big)
\]
for the link of a vertex~$x$ of~$H$.

\begin{proof}[Proof of Proposition~\ref{prop:crowded}]
	Let us fix distinct indices $i, j\in [t]$ as well as some $\alpha\in[\l]$.
	Setting
	\[
		I=\bigl\{(k, \beta)\in [t]\times [\l]\colon (i, j, k, \alpha, \beta)\in \fQcr(C, p)\bigr\}
	\]
	and writing~$H^{ij}_{\alpha}$ for the subhypergraph of~$H$ whose hyperedges lie
	in~$\ee$-triads~$P^{ijk}_{\alpha\beta}$ with $(k, \beta)\in I$, it suffices
	to show
	\begin{equation}
		\label{eq:crowded-goal}
		e(H^{ij}_{\alpha})\le 2\frac{m^2\sqrt n}{C\l}\,.
	\end{equation}
	In the special case $I=\vn$ there are no
	hyperedges in~$H^{ij}_{\alpha}$, and our claim is clear. We may therefore assume $I\ne\vn$ from now on.

	The dense graph counting lemma gives
	\[
		\big|\ccKee(P^{ijk}_{\alpha\beta})\big|
		\ge
		\Bigl(\frac 1{\l^2}-2\eps\Bigr)|V_i||V_j||V_k|\ge \frac{3m^3}{4\l^2}
	\]
	for every $k\in [t]\sm\{i, j\}$ and every $\beta\in [\l]$.
	Thus the definition of~$I$ immediately yields the lower bound
	\begin{equation}
		\label{eq:1721}
		e(H^{ij}_{\alpha})
		\ge
		\sum_{(k, \beta)\in I} e_{\ee}(H\,|\,P^{ijk}_{\alpha\beta})
		\ge
		|I|\cdot \frac{3m^3}{4\l^2}\cdot\frac{C}{\sqrt{n}}\,,
	\end{equation}
	which turns out to be useful later.

	For every $v\in V_j$ we set
	\[
		A_v= N_{P^{ij}_\alpha}(v)
		\qqand
		B_v= \bigdcup_{(k, \beta)\in I} N_{P^{jk}_\beta}(v)\,.
	\]
	The number of partite homomorphisms from~$K_{1,2}$ to
	the $(\eps, 1/\l)$-regular bipartite graph~$P^{ij}_\alpha$
	such that the central vertex is mapped to~$V_j$ is, again by
	the dense graph counting lemma,
	\begin{equation}
		\label{eq:1448}
		\sum_{v\in V_j}|A_v|^2
		\le
		\left(\frac 1{\l^2}+2\eps\right)|V_i|^2|V_j|
		\le
		\frac{5m^3}{4\l^2}\,.
	\end{equation}
	Next, we observe that the sum
	\[
		S=\sum_{v\in V_j} |B_v|
	\]
	can be rewritten as
	\[
		S
		=\sum_{(k, \beta)\in I} e(P^{jk}_\beta)
		\le
		|I|\cdot\left(\frac{1}{\l}+\eps\right)m^2
		\le
		|I|\cdot\frac{9m^2}{8\l}\,,
	\]
	whence
	\begin{equation}
		\label{eq:1737}
		\frac{e(H^{ij}_{\alpha})}{S}
		\overset{\eqref{eq:1721}}{\ge}
		\frac{2Cm}{3\l \sqrt{n}}\,.
	\end{equation}
	Since our assumptions on~$C$, $\tau$, $n$, and the bound $|V_0|\le \tau n$ imply
	\[
		Cm
		\ge
		m
		=
		\frac{n-|V_0|}t
		\ge
		\frac{3n}{4t}
		\ge
		\frac{3\cdot 32\l t\sqrt{n}}{4t}
		=
		24\l\sqrt{n}\,,
	\]
	we conclude from~\eqref{eq:1737} that
	\begin{equation}
		\label{eq:1739}
		\frac{e(H^{ij}_{\alpha})}{S}-1
		\ge
		\frac{Cm}{\l \sqrt{n}}\left(\frac 23-\frac 1{24}\right)
		=
		\frac{5Cm}{8\l \sqrt{n}}\,.
	\end{equation}
	
	\begin{figure}[ht]
	\centering
	\begin{tikzpicture}[line width=0.75pt, scale=0.92, every node/.style={scale=0.92}]
		\tikzset{
			cls/.style ={rounded corners=5pt, draw=black!35, line width=0.6pt},
			sub/.style ={rounded corners=4pt, draw=blue!55!black, line width=0.7pt,
			             fill=blue!70!black, fill opacity=0.10},
			pij/.style ={draw=blue!70!black},
			link/.style={draw=red!70!black, line width=1.0pt},
			lbl/.style ={blue!70!black},
		}
\foreach \k in {1,...,7} {
			\coordinate (a\k) at (5.02+0.52*\k, 4.2);   \coordinate (b\k) at (5.02+0.52*\k, 2.4);   \coordinate (c\k) at (-0.18+0.52*\k, 0.6);  \coordinate (d\k) at (5.02+0.52*\k, 0.6);   \coordinate (e\k) at (10.22+0.52*\k, 0.6);  }
		\coordinate (v)  at (7.10,2.4);
		\coordinate (yy) at (13.34,0.6);
		\coordinate (ys) at (0.34,0.6);
\tedge[fill opacity=0.12]{(v)}{(a4)}{(yy)}
		\tedge[fill opacity=0.12]{(v)}{(a6)}{(yy)}
\draw[cls] (5.2,3.8) rectangle (9.0,4.6);
		\draw[cls] (5.2,2.0) rectangle (9.0,2.8);
		\foreach \xo in {0,5.2,10.4} \draw[cls] (\xo,0.2) rectangle (\xo+3.8,1.0);
\draw[sub] (6.84,3.86) rectangle (8.40,4.54);
		\foreach \xo in {0.08,2.16,5.28,6.84,10.48,13.08}
			\draw[sub] (\xo,0.26) rectangle (\xo+1.04,0.94);
\foreach \p in {a4,a5,a6,c1,c2,c5,c6,d1,d2,d4,d5,e1,e2,e6,e7} \draw[pij] (v) -- (\p);
\draw[link] (ys) -- (a4);
		\draw[link] (ys) -- (a5);
\draw[blue!70!black, line width=1.0pt, fill=white] (v) circle[radius=3.4pt];
		\foreach \k in {1,...,7} {\hvertex{(a\k)} \hvertex{(b\k)}
		                          \hvertex{(c\k)} \hvertex{(d\k)} \hvertex{(e\k)}}
\node at (4.5,0.6) {$\cdots$};
		\node at (9.7,0.6) {$\cdots$};
		\node[anchor=west]      at (9.15,4.2)  {$V_i$};
		\node[anchor=west]      at (9.15,2.4)  {$V_j$};
		\node[anchor=west]      at (14.35,0.6) {$V_k$};
		\node[anchor=east, lbl] at (5.05,4.2)  {$A_v$};
		\node[anchor=east, lbl] at (-0.2,0.6)  {$B_v$};
		\node[anchor=south east] at (7.01,2.84) {$v$};
		\node[anchor=north] at (13.34,0.10) {$y$};
		\node[anchor=north] at (0.34,0.10)  {$y'$};
		\node[anchor=south] at (7.10,4.66)  {$x$};
		\node[anchor=south] at (8.14,4.66)  {$x'$};
	\end{tikzpicture}
	\caption{The vertex $v\in V_j$ has neighbourhood~$A_v$
		in~$V_i$, and the neighbourhoods $N_{P^{jk}_\beta}(v)$ with $(k,\beta)\in I$ make
		up~$B_v$. The blue edges at~$v$ are
		those of the graphs $P^{ij}_\alpha$ and~$P^{jk}_\beta$. The red edges at~$y'$ lead
		to $A_v\cap N_{L(v)}(y')$, and for~$y$ the two hyperedges $vxy$ and~$vx'y$ behind
		such edges are drawn as shaded triples instead. Here~$y$ and~$y'$ share only~$x$,
		because sharing~$x'$ as well would close the~$4$-cycle $xyx'y'$ in~$L(v)$.}
	\label{fig:paircount}
\end{figure}
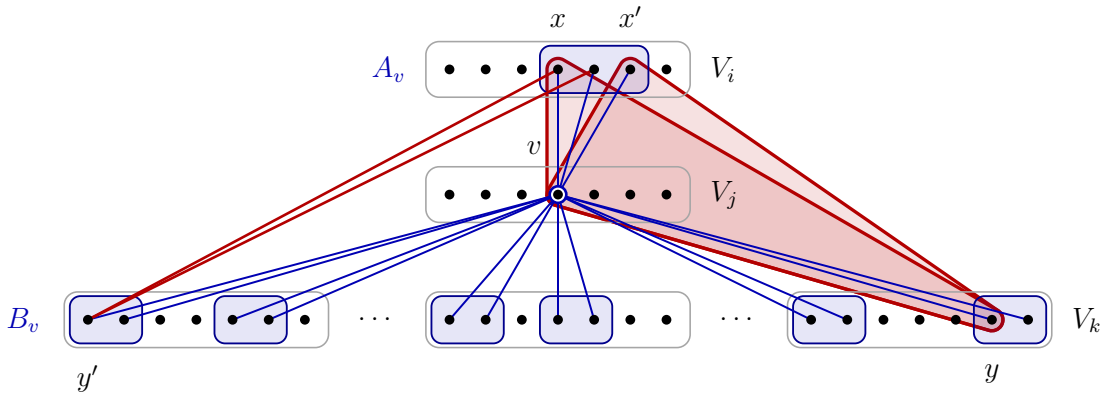

	After these preparations, the main point (which is illustrated in
	Figure~\ref{fig:paircount}) is that for each $v\in V_j$ we have
	\begin{equation}
		\label{eq:1459}
		\sum_{y\in B_v} \binom{|A_v\cap N_{L(v)}(y)|}2\le \binom{|A_v|}2\,,
	\end{equation}
	because both sides count pairs of vertices in~$A_v$.
	If some pair $\{x, x'\}\in A_v^{(2)}$ were counted twice on the left-hand side,
	there would have to be distinct vertices $y, y'\in B_v$ adjacent to both~$x$ and~$x'$ in~$L(v)$.
	But then~$L(v)$ would contain the~$4$-cycle~$xyx'y'$, contrary to~$H$ being~$K_{1,2,2}$-free.

	Having proved~\eqref{eq:1459}, we sum it over all $v\in V_j$. In view
	of~\eqref{eq:1448} this leads to
	\begin{equation}
		\label{eq:crowded-CS1}
		\sum_{v\in V_j}\sum_{y\in B_v} \binom{|A_v\cap N_{L(v)}(y)|}2 \le \frac{5m^3}{8\l^2}\,.
	\end{equation}
	The left-hand side is a sum of~$S$ binomial coefficients. Since
	\[
		\sum_{v\in V_j}\sum_{y\in B_v} \big|A_v\cap N_{L(v)}(y)\big|=e(H^{ij}_{\alpha})\,,
	\]
	a standard convexity argument shows
	\[
		e(H^{ij}_{\alpha})\left(\frac{e(H^{ij}_{\alpha})}{S}-1\right)
		=
		2S\binom{e(H^{ij}_{\alpha})/S}2
		\overset{\eqref{eq:crowded-CS1}}{\le}
		\frac{5m^3}{4\l^2}\,,
	\]
	which together with~\eqref{eq:1739} proves the desired estimate~\eqref{eq:crowded-goal} on $e(H^{ij}_{\alpha})$.
\end{proof}

\section{Sparse counting lemma}\label{sec:counting}

Before proving Proposition~\ref{prop:counting}, we briefly describe our strategy. The assumption that there are $\Omega(\l^7t^7)$
different~$\Cs$-conspiracies in the reduced quintuple system $\fQ(\delta, d, p)$
(Definition~\ref{def:reduced}) entails that there
are $\Omega(\l^4t^2)$ of them sharing some choice~$i_0$, $i_2$, $i_3$, $i_4$, $i_5$, $j_{23}$, $j_{34}$, and~$j_{45}$.
As it turns out, after fixing those indices we can still find~$\Omega_{t, \l}(n^4)$
copies $v_0\dots v_6$ of~$\Cs$ corresponding to these choices in the sense that
for $k\in \{0, 2, 3, 4, 5\}$ the vertex~$v_k$ will be in~$V_{i_k}$ and for $k\in\{2, 3, 4\}$
the pair $v_kv_{k+1}$ will belong to the bipartite graph~$P^{i_ki_{k+1}}_{j_{k, k+1}}$.

The detection of these copies proceeds in two major steps (see Figure~\ref{fig:strategy}).
First (cf.\ Lemma~\ref{lem:42}), from
a `typical' vertex $v_0\in V_{i_0}$ a sufficiently large proportion of the
pairs $v_2v_3\in E(P^{i_2i_3}_{j_{23}})$ can be reached by a tight path $v_0v_1v_2v_3$
in~$H$.
By the same argument, sufficiently many pairs $v_4v_5\in E(P^{i_4i_5}_{j_{45}})$ are reachable
by tight paths $v_0v_6v_5v_4$. As we shall see, for at least one half of the
vertices $v_0\in V_{i_0}$ this argument produces dense sets of reachable pairs
$Q_{23}(v_0)\subseteq V_{i_2}\times V_{i_3}$
and $Q_{45}(v_0)\subseteq V_{i_4}\times V_{i_5}$. Having found those, it remains
to observe in the second step that, by the regularity of~$H$, for sufficiently many pairs of pairs
$(v_2v_3, v_4v_5)\in Q_{23}(v_0)\times Q_{45}(v_0)$
the quadruple $v_2v_3v_4v_5$ is a tight path in~$H$ (cf.\ Lemma~\ref{lem:43}).
Both steps of the argument can be
studied in the following~$4$-partite context.

\begin{figure}[ht]
	\centering
	\begin{tikzpicture}[line width=0.75pt]
\foreach \k in {0,...,6}
		\coordinate (v\k) at ({90-360/7*\k}:2.4);
\foreach \a/\b/\c in {0/1/2, 1/2/3, 4/5/6, 5/6/0}
		{\tedge[draw=blue!70!black, fill=blue!70!black]{(v\a)}{(v\b)}{(v\c)}}
\foreach \a/\b/\c in {2/3/4, 3/4/5}
		{\tedge{(v\a)}{(v\b)}{(v\c)}}
		\draw[blue!70!black, line width=1.1pt] (v2) -- (v3) (v4) -- (v5);
		\draw[line width=1.1pt] (v3) -- (v4);
		\foreach \k in {0,...,6}
		{\hvertex{(v\k)}}
		\foreach \k in {0,...,6}
		\node at ({90-360/7*\k}:2.9) {$v_\k$};
		\node[blue!70!black] at ({90-360/7*2.5}:3.45) {$Q_{23}(v_0)$};
		\node at ({90-360/7*3.5}:3.45) {$P^{i_3i_4}_{j_{34}}$};
		\node[blue!70!black] at ({90-360/7*4.5}:3.45) {$Q_{45}(v_0)$};
		\node[blue!70!black] at (0,0.75)  {first step};
		\node[red!70!black]  at (0,-0.40) {second step};
	\end{tikzpicture}
	\caption{The two steps of the proof of Proposition~\ref{prop:counting}.}
	\label{fig:strategy}
\end{figure}
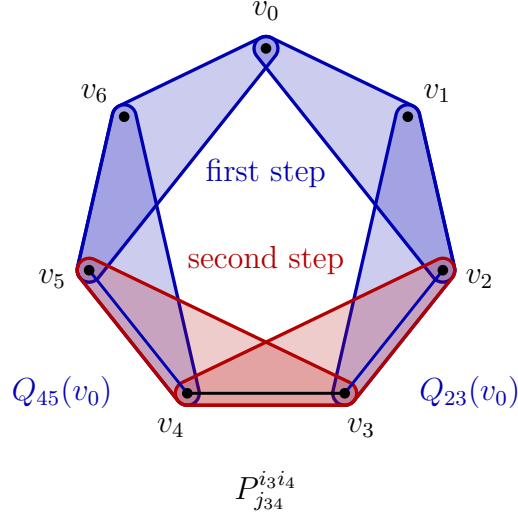

\begin{setup}\label{setup}
	Suppose that $P=(W\dcup X\dcup Y\dcup Z, E_P)$ is a~$4$-partite graph whose
	bipartite subgraphs $P[W, X]$, $P[X, Y]$, $P[Y, Z]$ are $(\eps, 1/\l)$-regular for
	some $\eps>0$ and $\l\ge 1$, and which has no further edges. Further,
	let~$\delta$, $d$, and~$p$ be positive reals, and let $H=(V, E)$ be a hypergraph
	with $V\supseteq W\dcup X\dcup Y\dcup Z$ that is $(\delta, p)$-regular with respect to
	the~$\ee$-partite graphs~$P[W, X, Y]$, $P[X, Y, Z]$
	and satisfies
	\[
		\min\big\{\dee\big(H\,|\,P[W, X, Y]\big)\,,\, \dee\big(H\,|\,P[X, Y, Z]\big)\big\}\ge dp\,.
	\]
\end{setup}

The first step amounts essentially to the following.

\begin{lemma}\label{lem:42}
	Given Setup~\ref{setup}, assume that $\eps\le \frac{1}{2^8\l^2}$
	and $\delta\le \min\{\frac{d}{2^4}, \frac{1}{2^7}\}$.
	Then there are at least $\frac78 |W|$ vertices $w\in W$ for which there
	are at least $\tfrac 14d^2p^2|X||Y||Z|/\l^3$ triples $(x, y, z)\in X\times Y\times Z$
	such that $wxyz$ is a path in~$P$ and a tight path in~$H$.
\end{lemma}

\begin{proof}
	All paths we are about to exhibit will meet~$Y$ in its subset
	\[
		Y'=\{y\in Y\colon |N_P(y)\cap Z|\ge \tfrac{15}{16}|Z|/\l\}\,.
	\]
	By the $(\eps, 1/\l)$-regularity of $P[Y, Z]$, we have
	\[
		\frac{|Y\sm Y'||Z|}{\l}-\eps |Y||Z|
		\le
		e(Y\sm Y', Z)
		\le \frac{15|Y\sm Y'||Z|}{16\l}\,,
	\]
	whence $|Y\sm Y'|\le 16\eps\l |Y|\le \tfrac{1}{16}|Y|$ and
	\[
		|Y'|\ge \tfrac{15}{16} |Y|\,.
	\]
	We are going to avoid the edges of $P[X, Y]$ belonging to the exceptional set
	\[
		R=\bigl\{xy'\in E_P\colon x\in X,\ y'\in Y', \tand |N_P(y')\cap N_H(x, y')\cap Z|<\tfrac 78 dp|Z|/\l\bigr\}\,.
	\]
	To show that this set is small, we work with
	the~$\ee$-partite graph
	\[
		Q=\big(X\dcup Y\dcup Z, R\cup E(P[Y, Z])\big)\,.
	\]
	Since the definitions of~$R$ and~$Y'$ imply
	\[
		|N_P(y')\cap N_H(x, y')\cap Z|<\tfrac{14}{15} dp \, |N_P(y')\cap Z|
	\]
	for every $xy'\in R$, we have $\dee(H\,|\,Q)<\frac{14}{15} dp\le \dee(H\,|\,P[X, Y, Z])-\delta p$.
	But~$H$ is $(\delta, p)$-regular with respect to $P[X, Y, Z]$, so this is only possible if
	\[
		|\ccKee(Q)|
		\le \delta|\ccKee(P[X, Y, Z])|
		\le \delta(\l^{-2}+2\eps)|X||Y||Z|
		\le 2\delta |X||Y||Z|/\l^2\,.
	\]
	As every pair in~$R$ contributes at least $\frac{15}{16}|Z|/\l$ triples
	to $\ccKee(Q)$, this entails
	\[
		|R|
		\le
		\frac{4\delta |X||Y|}{\l}
		\le
		\frac{|X||Y|}{32\l}\,.
	\]
	Consequently, the graph $S=P[X, Y']-R$ satisfies
	\[
		e(S)
		\ge
		\frac{|X||Y'|}{\l}-\eps|X||Y|-\frac{|X||Y|}{32\l}
		\ge
		\left(\frac{15}{16}-\frac 1{32}-\frac 1{32}\right)\frac{|X||Y|}{\l}
		=
		\frac{7|X||Y|}{8\l}\,.
	\]

	Next, we analyse the set
	\[
		X'=\bigl\{x\in X\colon d_S(x)\ge \tfrac 12|Y|/\l\bigr\}\,.
	\]
	Owing to
	\begin{align*}
		\frac{|X'||Y|}{\l}+\eps|X||Y|
		 & \ge
		e_P(X', Y)
		\ge
		e_S(X', Y)
		\ge
		\frac{7|X||Y|}{8\l}-e_S(X\sm X', Y) \\
		 & \ge
		\frac{7|X||Y|}{8\l}-\frac{|X\sm X'||Y|}{2\l}
		=
		\frac{(3|X|+4|X'|)|Y|}{8\l}
	\end{align*}
	we have $|X'|\ge \frac14(3-8\l\eps)|X|\ge\frac{11}{16}|X|$.

	We will now consider two exceptional subsets of~$W$. First,
	\[
		W'=\bigl\{w\in W\colon |N_P(w)\cap X'|\le \tfrac58 |X|/\l\bigr\}
	\]
	satisfies, by the same argument that bounded $|Y\sm Y'|$ above,
	$|W'|\le 16\eps\l|W|\le \tfrac{1}{16}|W|$.
	Second, we consider
	\[
		W''=\Bigl\{w\in W\sm W'\colon \big|\bigl\{xy\in E(S)\colon wx\in E_P \tand wxy\in E\bigr\}\big| \le \tfrac 27dp|X||Y|/\l^2\Bigr\}\,.
	\]

	We claim that~$W''$ is small as well. To see this, we look at the~$\ee$-partite subgraph $T\subseteq P[W, X, Y]$
	with $T[W, X]=P[W'', X']$ whose edges from~$X$ to~$Y$ are those of~$S$. Each $w\in W''$ has at least $\tfrac 58|X|/\l$ neighbours
	in~$X'$ (by $w\not\in W'$), which in turn have degree at least $\frac 12|Y|/\l$
	in~$S$. Therefore, every $w\in W''$ contributes at least $\frac{5}{16}|X||Y|/\l^2$
	triples to $\ccKee(T)$. On the other hand, by the definition of~$W''$, at most
	$\tfrac 27dp|X||Y|/\l^2$ of these triples support a hyperedge of~$H$.
	This shows $\dee(H\,|\,T)\le \frac{32}{35} dp< \dee(H\,|\,P[W, X, Y])-\delta p$. Owing to
	the $(\delta, p)$-regularity of~$H$ with respect to $P[W, X, Y]$, this tells us
	\[
		\tfrac{5}{16}|W''||X||Y|/\l^2
		\le
		|\ccKee(T)|
		\le
		\delta |\ccKee(P[W, X, Y])|
		\le
		2\delta |W||X||Y|/\l^2\,,
	\]
	whence $|W''|\le 8\delta |W|\le \tfrac{1}{16}|W|$.

	Altogether, there are at least $\frac 78|W|$ vertices in $W\sm(W'\cup W'')$,
	and it suffices to show that all of them have the property described in the lemma.
	For any $w\in W\sm (W'\cup W'')$, the definition of~$W''$ leads to at least
	$\frac 27 dp|X||Y|/\l^2$ pairs~$(x, y)\in X\times Y$ such that $wx, xy\in E_P$,
	$wxy\in E$, and, moreover, $xy\in E(S)$. By the definition of~$R$, for every
	such pair~$(x, y)$ there are at least $\frac 78dp|Z|/\l$ vertices $z\in Z$
	such that $yz\in E_P$ and $xyz\in E$. Thus there are indeed at least
	$\frac 14d^2p^2|X||Y||Z|/\l^3$ paths of the required kind.
\end{proof}

The second step in the plan laid out above is provided by the following result.

\begin{lemma}\label{lem:43}
	Given Setup~\ref{setup} and some $\zeta\in (0, 1]$, suppose
	$\eps\le \frac{\zeta^2}{2^8\l^2}$
	and $\delta\le \min\{\frac{d}{2^3}, \frac{\zeta^3}{2^8}\}$, and
	set $\theta=\frac{\zeta^4 d^2}{2^5\l^3}$.
	If a subgraph $Q\subseteq P[W, X]\cup P[Y, Z]$
	satisfies
	\[
		e_Q(W, X)\ge \zeta |W||X|/\l \qand e_Q(Y, Z)\ge \zeta |Y||Z|/\l\,,
	\]
	then there are at least $\theta p^2 |W||X||Y||Z|$ tight paths $wxyz$ in~$H$
	with $wx, yz\in E(Q)$ and $xy\in E_P$.
\end{lemma}

\begin{proof}
	We are going to take~$x$ and~$y$ from the sets
	\begin{align*}
		X' & =\bigl\{x\in X\colon d_Q(x)\ge \tfrac12 \zeta |W|/\l\bigr\}\qand
		Y' =\bigl\{y\in Y\colon d_Q(y)\ge \tfrac12 \zeta |Z|/\l\bigr\}\,,
	\end{align*}
	so we begin by estimating their sizes.
	Owing to $e_Q(W, X\sm X')\le \frac12 \zeta |W||X|/\l$
	we have $e_Q(W, X')\ge \frac12 \zeta |W||X|/\l$ and, therefore,
	\[
		\frac{\zeta|W||X|}{2\l}
		\le
		e_Q(W, X')
		\le
		e_P(W, X')
		\le
		\frac{|W||X'|}{\l}+\eps |W||X|\,,
	\]
	whence $|X'|\ge (\tfrac 12\zeta-\l\eps)|X|$.
	The same reasoning applies to~$Y'$, and since
	$\eps\le \frac{\zeta^2}{2^8\l^2}\le \frac{\zeta}{2^4\l}$ we obtain
	\[
		|X'|\ge \tfrac{7}{16}\zeta |X| \qand |Y'|\ge \tfrac{7}{16}\zeta |Y|\,.
	\]

	Next, we analyse certain `bad choices' for the edge~$xy$, namely
	\begin{align*}
		R_W & =\bigl\{xy\in E_P\colon x\in X',\ y\in Y,
		\tand |N_Q(x)\cap N_H(x, y)\cap W|< \tfrac 7{16} dp\zeta |W|/\l\bigr\} \\
		R_Z & =\bigl\{xy\in E_P\colon x\in X,\ y\in Y',
		\tand |N_Q(y)\cap N_H(x, y)\cap Z|< \tfrac 7{16} dp\zeta |Z|/\l\bigr\}\,.
	\end{align*}

	To show that these sets are small, we look at the~$\ee$-partite subgraph $T\subseteq P[W, X, Y]$ with
	$T[W, X]=Q[W, X]$ whose edges from~$X$ to~$Y$ are the pairs in~$R_W$. By the definition
	of~$X'$ each pair~$xy \in R_W$ satisfies
	\[
		|N_Q(x)\cap N_H(x, y)\cap W|< \tfrac 78 dp \,|N_Q(x)|\,,
	\]
	whence $\dee(H\,|\,T)<\frac 78 dp\le \dee(H\,|\,P[W, X, Y])-\delta p$.
	Owing to the $(\delta, p)$-regularity of~$H$ with respect to $P[W, X, Y]$
	this implies
	\[
		|\ccKee(T)|
		\le \delta |\ccKee(P[W, X, Y])|
		\le \tfrac43 \delta |W||X||Y|/\l^2\,.
	\]
	On the other hand, each pair $xy\in R_W$ participates in at
	least $\frac12 \zeta |W|/\l$ triples from~$\ccKee(T)$, so
	it follows that $|R_W|\le \frac83 \delta\zeta^{-1}|X||Y|/\l$.
	Arguing similarly for~$R_Z$ and using our upper bound on~$\delta$, we infer
	\[
		\max\{|R_W|, |R_Z|\}\le \tfrac 1{96}\zeta^2 |X||Y|/\l\,.
	\]

	Thus the graph
	\[
		S=P[X', Y']-(R_W\cup R_Z)
	\]
	has size at least
	\begin{align*}
		e(S)
		 & \ge
		e(X', Y')-|R_W|-|R_Z|
		\ge
		\frac{|X'||Y'|}{\l}-\eps |X||Y|-\frac{\zeta^2|X||Y|}{48\l} \\
		 & \ge
		\left(\frac{49}{256}-\frac 1{256}-\frac 1{48}\right)\frac{\zeta^2|X||Y|}{\l}
		=
		\frac{\zeta^2|X||Y|}{6\l}\,.
	\end{align*}

	Each edge $x'y'\in E(S)$ is in at least $(\frac 7{16}dp\zeta/\l)^2 |W||Z|$
	paths $wx'y'z$ of the desired kind, because, by the definition of~$R_W$, there
	are at least $\frac 7{16}dp\zeta |W|/\l$ possibilities for~$w$ and, similarly,
	there are at least $\frac 7{16}dp\zeta |Z|/\l $ possibilities for~$z$.
	In total, the number of
	paths provided by this argument is indeed at least
	\[
		\frac{\zeta^2|X||Y|}{6\l}\cdot \frac{7\zeta dp |W|}{16\l} \cdot \frac{7 dp\zeta |Z|}{16\l} >\theta p^2 |W||X||Y||Z|\,. \qedhere
	\]
\end{proof}

Finally, we deduce the sparse counting lemma for~$\Cs$ from Lemmata~\ref{lem:42} and~\ref{lem:43}.

\begin{proof}[Proof of Proposition~\ref{prop:counting}]
	Fix auxiliary constants
	\[
		\zeta=\frac{d^2\eta}{65} \qqand \theta=\frac{\zeta^4d^2}{2^5\l^3}\,.
	\]
	Since $d, \eta\le 1$, we have $\zeta\le 1$, and our assumptions on~$\delta$
	and~$\eps$ imply the hypotheses of Lemmata~\ref{lem:42} and~\ref{lem:43}
	for this value of~$\zeta$.

	Let~$H$ be a~$K_{1,2,2}$-free hypergraph on~$n$ vertices admitting
	a $(\tau,t,\l,\eps)$-equitable partition~$\ccP$ such that for
	some $p\ge n^{-1/2}$ the reduced quintuple system $\fQ(\delta, d, p)$
	contains at least $\eta\l^7 t^7$ distinct~$\Cs$-conspiracies.
	Let $V_0\dcup V_1\dcup\cdots\dcup V_t$ be the vertex partition
	underlying~$\ccP$, and set
	\[
		m
		=|V_1|
		=\dots
		=|V_t|
		=\frac{n-|V_0|}{t}\,.
	\]
	If such a $\Cs$-conspiracy
	$(i_0, \dots,i_6,j_{01},\dots, j_{60})$ satisfies $i_r=i_s$ for two distinct indices~$r$ and~$s$,
	then $\{r, s\}$ is either~$\{1, 6\}$ or of the form $\{k, k+3\}$, which leaves eight possibilities.
	Consequently, at most $8\l^7 t^6$ of those~$\Cs$-conspiracies have the property that two
	of their first seven indices coincide.
	By our lower bound $t\ge 16/\eta$, it follows that the set~$\fC$ of~$\Cs$-conspiracies
	$(i_0, \dots, j_{60})$ in $\fQ(\delta, d, p)$ with $|\{i_0, \dots, i_6\}|=7$ satisfies
	\[
		|\fC|
		\ge
		\eta \l^7 t^7-8\l^7 t^6
		=
		(\eta-8/t) \l^7 t^7
		\ge
		\tfrac 12\eta \l^7 t^7\,.
	\]
	The pigeonhole principle yields eight indices $i_0, i_2, i_3, i_4, i_5\in [t]$
	and $j_{23}, j_{34}, j_{45}\in [\l]$ common to at least $\frac 12\eta \l^4 t^2$ different
	$\Cs$-conspiracies in~$\fC$. This means
	\begin{equation}
		\label{eq:2347}
		(i_2, i_3, i_4, j_{23}, j_{34}), (i_3, i_4, i_5, j_{34}, j_{45})\in \fQ(\delta, d, p)\,,
	\end{equation}
	and that the set
	\[
		\fD=\bigl\{\big((i_1, j_{01}, j_{12}), (i_6, j_{56}, j_{60})\big) \in \bl[t]\times [\l]\times [\l]\br^2 \colon
			(i_0, \dots,i_6,j_{01},\dots, j_{60})\in \fC\bigr\}
	\]
	satisfies $|\fD|\ge \frac 12\eta \l^4 t^2$. The reason for this
	particular ordering of the entries of the sextuples in~$\fD$ is that~$\fD$ is
	`almost' the Cartesian product of the set of possible triples~$(i_1, j_{01}, j_{12})$
	with the set of possible triples~$(i_6, j_{56}, j_{60})$. In fact, the only difference
	between~$\fD$ and this product is that the sextuples in~$\fD$ need to
	satisfy $i_1\ne i_6$. The next step in the argument allows us to process the triples
	$(i_1, j_{01}, j_{12})$ and $(i_6, j_{56}, j_{60})$ independently.

	For every partition $[t]\sm\{i_0, i_2, i_3, i_4, i_5\}=T_1\dcup T_6$ the sets
	\begin{align*}
		M_1(T_1) & =\bigl\{(i_1, j_{01}, j_{12})\in T_1\times [\l]^2\colon
		(i_0, i_1, i_2, j_{01}, j_{12}), (i_1, i_2, i_3, j_{12}, j_{23})\in \fQ(\delta, d, p)\bigr\}\,, \\
		M_6(T_6) & =\bigl\{(i_6, j_{56}, j_{60})\in T_6\times [\l]^2\colon
		(i_4, i_5, i_6, j_{45}, j_{56}), (i_5, i_6, i_0, j_{56}, j_{60})\in \fQ(\delta, d, p)\bigr\}
	\end{align*}
	satisfy $M_1(T_1)\times M_6(T_6)\subseteq\fD$. Conversely, every sextuple in~$\fD$
	belongs to this product for one quarter of the partitions
	$[t]\sm\{i_0, i_2, i_3, i_4, i_5\}=T_1\dcup T_6$.
	Consequently, we can fix one such partition such that the sets $M_1=M_1(T_1)$
	and $M_6=M_6(T_6)$ satisfy
	\begin{equation}
		\label{eq:1337}
		|M_1||M_6|
		\ge \tfrac14 |\fD|
		\ge \tfrac 18\eta \l^4 t^2\,.
	\end{equation}

	Setting $A_1=\bigdcup_{i_1\in T_1} V_{i_1}$ and $A_6=\bigdcup_{i_6\in T_6} V_{i_6}$,
	we will now locate $\lambda n^4$ copies of~$\Cs$ in~$H$ such that their
	vertices $v_0,\dots, v_6$ are taken from the mutually disjoint sets
	\[
		V_{i_0}, A_1, V_{i_2}, V_{i_3}, V_{i_4}, V_{i_5}, \text{ and } A_6\,.
	\]

	From~\eqref{eq:1337} and the trivial bounds
	$\max\{|M_1|, |M_6|\}\le \l^2 t$ we conclude
	\[
		\min\{|M_1|, |M_6|\}\ge \tfrac 18\eta \l^2 t\,.
	\]

	\begin{clm}\label{clm:44}
		For at least one half of the vertices $v_0\in V_{i_0}$ the sets
		\begin{align*}
			Q_{23}(v_0)
			  =\bigl\{(v_2, v_3)\in V_{i_2}\times V_{i_3}\colon{}
			 & v_2v_3\in E(P^{i_2i_3}_{j_{23}})\text{ and}\\
			 & \text{for some }v_1\in A_1\text{ there is a tight path }v_0v_1v_2v_3\text{ in }H\bigr\}\,, \\
			 \intertext{and}
			Q_{45}(v_0)
			  =\bigl\{(v_4, v_5)\in V_{i_4}\times V_{i_5}\colon{}
			 & v_4v_5\in E(P^{i_4i_5}_{j_{45}})\text{ and} \\
			 & \text{for some }v_6\in A_6\text{ there is a tight path }v_0v_6v_5v_4\text{ in }H\bigr\}
		\end{align*}
		satisfy $\min\{|Q_{23}(v_0)|, |Q_{45}(v_0)|\}\ge \zeta m^2/\l$.
	\end{clm}

	\begin{proof}
		By symmetry, it suffices to show that at least three quarters of the vertices
		$v_0\in V_{i_0}$ satisfy $|Q_{23}(v_0)|\ge \zeta m^2/\l$.
		To this end, we observe that for every triple $s=(i_1, j_{01}, j_{12})\in M_1$
		Lemma~\ref{lem:42}
		applied to the graph
		$P_s=P^{i_0i_1}_{j_{01}}\cup P^{i_1i_2}_{j_{12}}\cup P^{i_2i_3}_{j_{23}}$
		shows that the set $W(s)$ of all~${v_0\in V_{i_0}}$ such that there are at least
		$\frac14d^2p^2m^3/\l^3$ triples $(v_1, v_2, v_3)$ for which $v_0v_1v_2v_3$
		is a path in~$P_s$ and a tight path in~$H$ satisfies $|W(s)|\ge\frac 78m$.
		Now a simple double counting argument applied to the set
		\[
			X_{i_0} = \big\{v_0\in V_{i_0}\colon \text{there are at least $|M_1|/2$ triples $s\in M_1$ with $v_0\in W(s)$}\big\}
		\]
		yields
		\[
			\tfrac78 |M_1|m
			\le \sum_{s\in M_1} |W(s)|
			\le
			(m-|X_{i_0}|)\tfrac 12|M_1|+|X_{i_0}||M_1|\,,
		\]
		whence $\frac 34m\le |X_{i_0}|$. So it suffices to show that we have
		$|Q_{23}(v_0)|\ge \zeta m^2/\l$ for each $v_0\in X_{i_0}$.

		Let $v_0\in X_{i_0}$. Since at least $\frac12 |M_1|$ triples $s\in M_1$
		satisfy $v_0\in W(s)$, they contribute
		\[
			\frac{|M_1|}2\cdot \frac{d^2p^2m^3}{4\l^3}
			\ge
			\frac{\eta p^2d^2m^3t}{64\l}
			\ge
			\frac{d^2 \eta m^2}{65\l}
		\]
		tight paths $v_0v_1v_2v_3$ in~$H$, where we used
		$p^2mt\ge mt/n\ge 1-\tau\ge \frac{64}{65}$.
		These paths are mutually distinct, since~$\ccP$ partitions the pairs
		between any two vertex classes, whence the triple~$s$ can be recovered
		from such a path.
		So it remains to observe that no two such paths can end in the same
		pair~$v_2v_3$. This is because if $v_0v_1v_2v_3$ and
		$v_0v'_1v_2v_3$ were two tight paths in~$H$ with $v_1\ne v'_1$,
		then~$v_0v_1v_3v'_1$ would be a~$4$-cycle in the link of~$v_2$.
	\end{proof}

	Now for each $v_0\in V_{i_0}$ as obtained by Claim~\ref{clm:44}, we apply
	Lemma~\ref{lem:43} to the graph
	$P^{i_2i_3}_{j_{23}}\cup P^{i_3i_4}_{j_{34}}\cup P^{i_4i_5}_{j_{45}}$,
	for which Setup~\ref{setup} is guaranteed by~\eqref{eq:2347}, and to the
	subgraph of $P^{i_2i_3}_{j_{23}}\cup P^{i_4i_5}_{j_{45}}$ whose edges are the
	pairs in $Q_{23}(v_0)\cup Q_{45}(v_0)$. This yields at least
	$\theta m^4p^2$ quadruples $(v_2, v_3, v_4, v_5)$ such that $(v_2, v_3)\in Q_{23}(v_0)$,
	$(v_4, v_5)\in Q_{45}(v_0)$, and $v_2v_3v_4v_5$ is a tight path in~$H$.

	Summarising this discussion, we have found at least $\frac m2$ vertices $v_0\in V_{i_0}$
	for which there are at least $\theta m^4p^2$ copies $v_0v_1v_2v_3v_4v_5v_6$ of~$\Cs$ in~$H$. Hence, this argument yields at least
	\[
		\frac{\theta p^2m^5}{2}
		\ge
		\frac{\theta p^2(1-\tau)^5n^5}{2t^5}
		\ge
		\frac{d^{10}\eta^4 n^4}{2^{31}\l^3 t^5}
		=
		\lambda n^4
	\]
	copies of~$\Cs$.
\end{proof}

\section{Removal of conspiracies}
\label{sec:reduced-removal}

In this section we prove Proposition~\ref{prop:reducedremoval}. We require a
partite version of the hypergraph removal lemma (Theorem~\ref{thm:graphremovallemma}), which deals
with the following kind of situation. Suppose that an~$r$-uniform
hypergraph~$F$ with vertex set~$[s]$ and an~$s$-partite, $r$-uniform
hypergraph~${H=(V,E)}$ with distinguished vertex partition
$V=V_1\dcup\cdots\dcup V_s$ are given.
By a \emph{transversal copy} of~$F$ in~$H$ we mean a homomorphism~$\psi$
from~$F$ to~$H$ satisfying $\psi(i)\in V_i$ for every~${i\in [s]}$.

\begin{thm}[Partite removal lemma]\label{thm:partite}
	Given an~$r$-uniform hypergraph~$F$ with vertex set~$[s]$ and $\eps>0$, there
	is some $\delta>0$ such that the following holds. 
	
	If an~$s$-partite, $r$-uniform
	hypergraph~$H=(V,E)$
	with distinguished vertex partition $V=V_1\dcup\cdots\dcup V_s$ contains at most
	$\delta\prod_{i=1}^s |V_i|$ transversal copies of~$F$, then there is a
	set of hyperedges $E'\subseteq E(H)$ such that $H-E'$ contains no transversal copies of~$F$
	and every $e\in E(F)$ satisfies
	$\big|\big\{f\in E'\colon f\subseteq \bigdcup_{i\in e} V_i\big\}\big|\le \eps\prod_{i\in e} |V_i|$. \qed
\end{thm}

This version of the removal lemma was stated explicitly by Tao~\cite{Tao}*{Corollary 1.14}.
Note that it can be proved by the standard application of the hypergraph
regularity method that is familiar from the non-partite case.
Depending on how one organises the argument, it may seem that an additional assumption of the
form $|V_i|\ge \Omega_\eps(1)$ is needed. This hypothesis, however, can always be eliminated
either by adjusting~$\delta$ or by employing a standard blow-up argument. We omit the
details.

\begin{proof}[Proof of Proposition~\ref{prop:reducedremoval}]
	Let~$F$ be the~$5$-uniform hypergraph with fourteen vertices
	\[
		V(F)=\{0, 1, 2, 3, 4, 5, 6, (0,1), (1,2), (2, 3), (3, 4), (4, 5), (5, 6), (6, 0)\}
	\]
	and six hyperedges
	\[
		E(F) = \Bigl\{\bigl\{k-1, k, k+1, \bl k-1, k\br, \bl k, k+1\br\bigr\}\colon k\in\{1, 2, 3, 4, 5, 6\}\Bigr\}\,,
	\]
	where here and throughout the argument indices are read modulo~$7$.
	Identifying $V(F)$ with~$[14]$ in an arbitrary way, we apply
	Theorem~\ref{thm:partite} to~$F$ and $\eps=\rho/6$ and obtain
	some $\delta>0$. We claim that $\cdr=\delta$ has the desired property.

	To verify this, consider a $(t, \l)$-quintuple system~$\fQ$ containing at
	most $\cdr \l^7 t^7$ different~$\Cs$-conspiracies. We construct a
	$14$-partite~$5$-uniform hypergraph~$\cQ$ as follows:
	\begin{enumerate}[label=\rmlabel]
		\item Each $k\in \ZZ/7\ZZ$ contributes two vertex classes
		      \[
			      V_k=[t]\times \{k\} \qand V_{k, k+1}=[\l]\times \{(k, k+1)\}
		      \]
		      to~$\cQ$.
		\item For every quintuple $(u, v, w, \alpha, \beta)\in \fQ$
		      and every non-zero $k\in \ZZ/7\ZZ$, we put the hyperedge
		      \[
			      \bigl\{(u, k-1), (v, k), (w, k+1), \bl\alpha, (k-1, k)\br, \bl\beta, (k, k+1)\br\bigr\}
		      \]
		      into $E(\cQ)$.
	\end{enumerate}

	Every~$\Cs$-conspiracy
	\[
		(i_0, i_1, i_2, i_3, i_4, i_5, i_6, j_{01}, j_{12}, j_{23}, j_{34}, j_{45}, j_{56}, j_{60})
	\]
	in~$\fQ$ gives rise to a transversal copy of~$F$ in~$\cQ$ whose vertices are
	all $(i_k, k)\in V_k$ and all $\bl j_{k, k+1}, (k, k+1)\br\in V_{k, k+1}$
	with $k\in \ZZ/7\ZZ$. Conversely, every transversal copy of~$F$ in~$\cQ$ corresponds
	to a $\Cs$-conspiracy in~$\fQ$. Therefore, $\cQ$ contains at most $\cdr \l^7t^7$
	transversal copies of~$F$. Since the fourteen vertex classes of~$\cQ$ satisfy
	$\prod_{i=1}^{14}|V_i|=\l^7t^7$, our choice of~$\cdr$ yields a set
	$E'\subseteq E(\cQ)$ such that $\cQ-E'$ has no transversal copies of~$F$
	and for every non-zero $k\in\ZZ/7\ZZ$ the set
	\begin{multline*}
		\fQ'_k=\Bigl\{(u, v, w, \alpha, \beta)\in [t]^3\times [\l]^2\colon \\
		\bigl\{(u, k-1), (v, k), (w, k+1), \bl\alpha, (k-1, k)\br,
		\bl\beta, (k, k+1)\br\bigr\}\in E'\Bigr\}
	\end{multline*}
	satisfies $|\fQ'_k|\le \rho \l^2t^3/6$. Since every hyperedge of~$\cQ$ arises
	from a quintuple in~$\fQ$, the set~$\fQ'=\bigcup_{k\ne 0}\fQ'_k$ is a subset
	of~$\fQ$ of size $|\fQ'|\le \rho\l^2 t^3$.

	Finally, if $\fQ\sm \fQ'$ contained a $\Cs$-conspiracy, then for every
	non-zero $k\in\ZZ/7\ZZ$ the quintuple $(i_{k-1}, i_k, i_{k+1}, j_{k-1, k}, j_{k, k+1})$
	would lie outside~$\fQ'_k$, so that none of the six corresponding hyperedges
	of~$\cQ$ belongs to~$E'$. Thus this conspiracy would still yield a transversal
	copy of~$F$ in $\cQ-E'$, which is impossible.
\end{proof}

\section{Proof sketch of the sparse regularity lemma}\label{app:regularity}
The proof of Proposition~\ref{prop:regularity} follows along the lines of the proof of the hypergraph regularity lemma in~\cite{FR02} and uses ideas from~\cite{scott}.

Since the hypergraph~$H$ in Proposition~\ref{prop:regularity} is sparse, we
measure the progress of the regularisation by an index taken from~\cite{scott}.
To define it, fix $D\ge 1$ and consider the function
\begin{align*}
	\phi(x) = \begin{cases}
		          x^2
		           & \text{ if } x\le 2D, \\
		          4D(x-D)
		           & \text{ if } x\ge 2D.
	          \end{cases}
\end{align*}
Note that~$\phi$ is convex (see~\cite{scott}) and that $\phi(x) \le 4Dx$ for
every $x\ge 0$.
For~$\delta$ given by Proposition~\ref{prop:regularity} we will set 
\[
	D = \frac{48}{\delta^2}
\] 
for the remainder of this section.
\begin{dfn}[Index]
	Given a hypergraph~$H$ on~$n$ vertices, $p\in (0,1)$, and any $(\tau,t,\l)$-partition of the pairs
	\[
		\ccP = \big\{P^{ij}_\alpha\colon \alpha\in [\l]\tand ij\in [t]^{(2)}\big\}
	\]
	and
	\[
		\Triad(\ccP) = \big\{(P^{ij}_\alpha,P^{jk}_\beta)\colon i,j,k\in[t]\text{ are pairwise distinct} \tand \alpha,\beta\in [\l]\big\}\,,
	\]
	let
	\[
		\Ind(H;\ccP) = \frac{1}{6n^3}\sum_{P\in \Triad(\ccP)}\phi\left(\frac{\dee(H\,|\,P)}{p}\right)|\ccKee(P)|\,.
	\]
\end{dfn}
The factor~$1/6$ accounts for the ordering of the~$\ee$-triads: every hyperedge whose vertices lie in three distinct non-exceptional vertex classes occurs in exactly six members of~$\Triad(\ccP)$. In particular,
\[
	\sum_{P\in\Triad(\ccP)}e_{\ee}(H\,|\,P)\le 6e(H)\,.
\]
Since $\phi(x)\le 4Dx$ and, for~$n$ sufficiently large, every~$K_{1,2,2}$-free hypergraph on~$n$ vertices has at most~$n^{5/2}$ hyperedges, for any~$\ccP$ we have
\[
	\Ind(H;\ccP)
	\le \frac{4D}{6pn^3}\sum_{P\in\Triad(\ccP)}e_{\ee}(H\,|\,P)
	\le 4D \frac{e(H)}{pn^{3}}
	\le 4D\,.
\]
The proof of Proposition~\ref{prop:regularity} follows by iterating the following index increment lemma.
\begin{lemma}
	\label{lem:indinc}
	Let $\tau>0$ and $\delta\in(0,1/4]$, and let
	$\eps\colon \NN\lra (0,1]$ satisfy $\eps(x)\le(8x^2)^{-1}$
	for every $x\in\NN$. Let~$t$, $\l$ be positive integers with
	$\tau+\delta2^{-\l}\le 1/4$. Then there exist positive integers~$T$, $L$, and~$N_0$ such that for every $n\ge N_0$ and $p = n^{-1/2}$
	the following holds.

	Given a~$K_{1,2,2}$-free~$3$-uniform hypergraph~$H$ and a $(\tau,t,\l,\eps(\l))$-equitable partition~$\ccP$ with at least $\delta\l^2t^3$ $\ee$-triads that are not $(\delta,p)$-regular with respect to~$H$, there exists a $(\tau+\delta\cdot 2^{-\l},t_1,\l_1,\eps(\l_1))$-equitable partition~$\ccP_1$ such that
	\[
		\Ind(H;\ccP_1) \ge \Ind(H;\ccP) + \frac{\delta^4}{48} - \frac{20D}{\sqrt{\l2^{2t\l}}}
	\]
	and $t\le t_1\le T$ and $\l< \l_1\le L$.
\end{lemma}
\begin{proof}[Proof of Proposition~\ref{prop:regularity} assuming Lemma~\ref{lem:indinc}]
	By decreasing~$\tau$, $\delta$, and the values of~$\eps$, if necessary, we may assume
	that $\delta\le\tau\le 1/4$ and $\eps(x)\le (8x^2)^{-1}$ for every $x\in\NN$.
	We may also increase~$\l_0$ so that $\l_0>1$ and
	\begin{align}
		\label{eqn:regproofl0islarge}
		\frac{20D}{\sqrt{\l_0}} \le \frac{\delta^4}{96}\,.
	\end{align}

	Set $K=\lceil 2^{11}3^2\delta^{-6}\rceil+1$. Starting with
	$(t_0,\l_0,\tau/2)$, form the rooted tree of parameter triples
	reachable in at most~$K$ applications of Lemma~\ref{lem:indinc}, with one
	child for every admissible integer pair $(\l_1,t_1)$ within the bounds
	supplied at its parent. This tree is finite, since it has depth at most~$K$
	and every node has finitely many children. Let~$T_0$, $L_0$, and~$n_0$ be
	the maxima of all bounds~$T$, $L$, and~$N_0$ occurring in the tree, and
	increase~$n_0$
	if necessary so that an initial
	$(\tau/2,t_0,\l_0,\eps(\l_0))$-equitable partition exists.

	Now let $n\ge n_0$, set $p=n^{-1/2}$, and let~$H$ be a~$K_{1,2,2}$-free
	$3$-uniform hypergraph on~$n$ vertices. Choose such an initial partition~$\ccP_0$
	and set $\tau_0=\tau/2$. Suppose that after~$i$ steps we have constructed an
	$(\tau_i,t_i,\l_i,\eps(\l_i))$-equitable partition~$\ccP_i$, where
	\[
		\tau_i=\frac{\tau}{2}+\delta\sum_{h=0}^{i-1}2^{-\l_h}
	\]
	(with the sum empty when $i=0$). If at most $\delta\l_i^2t_i^3$ of the $\ee$-triads
	fail to be $(\delta,p)$-regular with respect to~$H$, then we are done.
	Otherwise, since the~$\l_h$ are strictly increasing integers and $\l_0>1$,
	we have $\l_h\ge h+2$, and hence
	\begin{align*}
		\tau_i+\delta2^{-\l_i}
		 & =\frac{\tau}{2}+\delta\sum_{h=0}^{i}2^{-\l_h}
		 \le \frac{\tau}{2}+\delta\sum_{h=0}^{i}2^{-(h+2)}
		=\frac{\tau}{2}+\delta\bigl(\tfrac12-2^{-(i+2)}\bigr)
		<\frac{\tau}{2}+\frac{\delta}{2}
		\le\tau
		\le\frac14\,.
	\end{align*}
	Thus Lemma~\ref{lem:indinc} applies and yields an
	$(\tau_{i+1},t_{i+1},\l_{i+1},\eps(\l_{i+1}))$-equitable partition~$\ccP_{i+1}$,
	where
	\[
		\tau_{i+1}=\tau_i+\delta2^{-\l_i},
		\qquad \l_{i+1}>\l_i,
		\qquad t_{i+1}\ge t_i\,.
	\]
	The displayed formula for~$\tau_i$ is therefore preserved at the next step.
	In particular, $\tau_i<\tau$ throughout the process.

	Moreover, \eqref{eqn:regproofl0islarge} and Lemma~\ref{lem:indinc} show that
	every unsuccessful step increases the index by at least
	\begin{align*}
		\Ind(H;\ccP_{i+1})-\Ind(H;\ccP_i)
		 & \ge \frac{\delta^4}{48}
		-\frac{20D}{\sqrt{\l_i2^{2t_i\l_i}}}
		\ge \frac{\delta^4}{48}-\frac{20D}{\sqrt{\l_0}}
		\ge \frac{\delta^4}{96}\,.
	\end{align*}
	As $\Ind(H;\ccP_i)\le 4D$ for every~$i$, there can be at most
	\[
		\frac{4D}{\delta^4/96}=2^{11}3^2\delta^{-6}
	\]
	unsuccessful steps. Since $K>2^{11}3^2\delta^{-6}$, the process terminates in fewer than~$K$ applications with a partition~$\ccP_i$
	for which at most $\delta\l_i^2t_i^3$ of the $\ee$-triads fail to be
	$(\delta,p)$-regular with respect to~$H$. The bound $\tau_i<\tau$ shows that
	this is a $(\tau,t_i,\l_i,\eps(\l_i))$-equitable partition.
	By the construction of~$T_0$ and~$L_0$, the parameters of~$\ccP_i$ also satisfy $t_i\le T_0$ and
	$\l_i\le L_0$. This proves Proposition~\ref{prop:regularity}.
\end{proof}
We now outline the proof of Lemma~\ref{lem:indinc}. We will need the following definition.
\begin{dfn}[Refinement of partitions]
	Given a ground set~$S$ and families of subsets $\ccA, \ccB$ of~$S$, we say that~$\ccA$ is finer than~$\ccB$ if for every $A\in \ccA$, there exists $B\in \ccB$ such that $A\subseteq B$. We say that~$\ccA$ refines~$\ccB$ if further $\bigcup \ccA = \bigcup\ccB$.
\end{dfn}

As in~\cite{FR02}, we obtain the equitable partition~$\ccP_1$
in three steps, each refining the partition of the pairs in~$V(H)$.
The steps are summarised in Figure~\ref{fig:regscheme} and formalised
in Facts~\ref{fact:step1} (Step 1), \ref{fact:step21} (Step 2),
and~\ref{fact:step3} (Step 3) below.
\begin{figure}[ht]
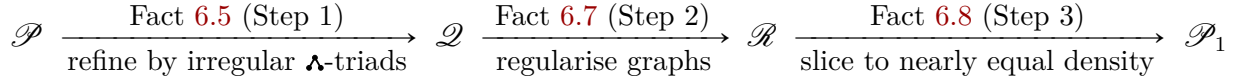

	\centering
		$\ccP\
		\xrightarrow[\text{\small refine by irregular~$\ee$-triads}]{\text{\small Fact~\ref{fact:step1} (Step 1)}}
		\ \ccQ\
		\xrightarrow[\text{\small regularise graphs}]{\text{\ \small Fact~\ref{fact:step21} (Step 2)\ }}
		\ \ccR\
		\xrightarrow[\text{\small slice to nearly equal density}]{\text{\small Fact~\ref{fact:step3} (Step 3)}}
		\ \ccP_1$
	\caption{The three refinement steps in the proof of Lemma~\ref{lem:indinc}.}
	\label{fig:regscheme}
\end{figure}

Before stating these facts, we fix the setup we will work with.
\begin{setup}
	Let $\tau>0$, $\delta\in(0,1/4]$, a function $\eps\colon \NN\lra (0,1]$ satisfying $\eps(x)\le (8x^2)^{-1}$ for every $x\in\NN$, and positive integers $t,\l$ be as given in the statement of Lemma~\ref{lem:indinc}. Throughout this proof, we may enlarge~$N_0$ to meet the finitely many lower bounds stated below. All of them depend only on the displayed parameters. Let~$H$ be a~$K_{1,2,2}$-free hypergraph on $n\ge N_0$ vertices.

	Let~$\ccP$ be a $(\tau,t,\l,\eps(\l))$-equitable partition
	(Definition~\ref{def:equit-partition}) with vertex partition
	\begin{align}
		\label{eqn:initialpartitionP}
		V(H) = V_0 \dcup V_1 \dcup V_2\dcup \cdots \dcup V_t\,.
	\end{align}
	We record a lower bound on the number of paths in every~$\ee$-triad of~$\ccP$. Put
	\[
		m=|V_1|=\cdots=|V_t|\,.
	\]
	Since $\tau\le 1/4$, we have
	$m\ge 3n/(4t)$. For a fixed~$\ee$-triad $P=(P^{ij}_\alpha,P^{jk}_\beta)$ and
	$y\in V_j$, put
	\[
		a_y=\deg_{P^{ij}_\alpha}(y) \qand b_y=\deg_{P^{jk}_\beta}(y)\,.
	\]
	Writing $x_+=\max\{x, 0\}$ for the positive part of a real number~$x$,
	the definition of $(\eps(\l),1/\l)$-regularity, applied to the set of vertices
	whose degree is below~$m/\l$, gives
	\[
		\sum_{y\in V_j}\Bigl(\frac m\l-a_y\Bigr)_+\le \eps(\l)m^2
		\qand
		\sum_{y\in V_j}\Bigl(\frac m\l-b_y\Bigr)_+\le \eps(\l)m^2\,.
	\]
	For every~$y$,
	\[
		a_yb_y\ge \frac{m^2}{\l^2}-\frac m\l\Bigl(\Bigl(\frac m\l-a_y\Bigr)_++\Bigl(\frac m\l-b_y\Bigr)_+\Bigr)\,.
	\]
	Consequently, using $\eps(\l)\le 1/(8\l^2)$, we obtain
	\begin{align}
		\label{eqn:triadpathlowerbound}
		|\ccKee(P)|
		 & =\sum_{y\in V_j}a_yb_y
		\ge \Bigl(\frac1{\l^2}-\frac{2\eps(\l)}\l\Bigr)m^3
		\ge \frac{3m^3}{4\l^2}
		\ge \frac{n^3}{4\l^2t^3}\,.
	\end{align}
	Further, we assume that there are at least $\delta\l^2t^3$ distinct $\ee$-triads $P^{ijk}_{\alpha\beta}\in \Triad(\ccP)$ that fail to be $(\delta,p)$-regular with respect to~$H$. Finally, let
	\begin{align*}
		\l' = \l 2^{2t\l} \qand \l_1 = 4(\l')^2\,.
	\end{align*}
\end{setup}
The refinement proceeds in three steps, each increasing the index: Step~1 separates the~$\ee$-triads that fail to be regular, Step~2 regularises the resulting bipartite graphs, and Step~3 slices them into pieces of nearly equal density. Only the last step is delicate. We now state these three facts and complete the proof of Lemma~\ref{lem:indinc} assuming them. Then we sketch the proofs of Facts~\ref{fact:step1} and~\ref{fact:step3}, while Fact~\ref{fact:step21} will be an immediate consequence of a multicolour graph regularity lemma (see Theorem~\ref{thm:graphreg} below).
\begin{fact}[Step 1]
	\label{fact:step1}
	There exists a $(\tau,t,\l')$-partition
	\[\ccQ = \big\{Q^{ij}_\alpha\colon \alpha\in [\l'] \tand ij\in [t]^{(2)}\big\}\]
	on~$V(H)$ such that~$\ccQ$ refines~$\ccP$, the associated vertex partition of~$\ccQ$ is~\eqref{eqn:initialpartitionP}, and further
	\[
		\Ind(H;\ccQ) \ge \Ind(H;\ccP) + \frac{\delta^4}{48}\,.
	\]
\end{fact}
Having obtained the partition~$\ccQ$ as above, we will now regularise the $|\ccQ| = \binom{t}{2}\l'$ graphs simultaneously and obtain a new vertex partition satisfying the properties stated in Fact~\ref{fact:step21} below. We first state the multicolour regularity lemma for graphs, which we will use.
\begin{thm}[Multicolour graph regularity]
	\label{thm:graphreg}
	For all integers~$h$, $t\geq 1$, and every $\eps > 0$ there exists an integer $\SZRL(\eps, h, t)$ such that given~$h$ graphs $G_1,\dots, G_h$ on the vertex set
	\[
		V = V_0\dcup V_1\dcup \cdots \dcup V_t\,,
	\]
	with $|V_1| = \cdots = |V_t|$, there exist an integer~$t'$ with $tt'<\SZRL(\eps,h,t)$ and a partition
	\[
		V
		= W_0 \dcup \bigdcup_{i=1}^t\left(\bigdcup_{j=1}^{t'} W_{i,j}\right)
		= W_0 \dcup W_1\dcup W_2\dcup \cdots \dcup W_{tt'}\,,
	\]
	such that
	\begin{enumerate}[label=\alabel]
		\item We have $W_{i,j}\subseteq V_i$ for each $i\in [t]$ and $j\in [t']$.
		\item For each $i\in [t]$, we have $V_i\subseteq W_0\dcup\bigdcup_{j\in [t']}W_{i,j}$.
		\item The sets~$W_{i,j}$ have equal size.
		\item\label{it:Reg3} The exceptional class satisfies $V_0\subseteq W_0$ and $|W_0\sm V_0|\leq tt'$.
		\item For all but at most $\eps(tt')^2$ choices of $(i_1,i_2,j_1,j_2)$ with $i_1i_2\in [t]^{(2)}$ and~$j_1$, $j_2\in [t']$,
		      \[
			      G_r[W_{i_1,j_1}, W_{i_2,j_2}] \text{ is } \eps\text{-regular} \text{ for every $r\in [h]$}\,.
		      \]
	\end{enumerate}
\end{thm}
Applying Theorem~\ref{thm:graphreg} to the graphs in~$\ccQ$ with
$\eps' = (\eps(\l_1))^6$ and $h = |\ccQ| = \binom{t}{2}\l'$,
we obtain the following\footnote{This choice yields the explicit error bounds displayed
	in the proof of Fact~\ref{fact:step3} immediately before~\eqref{eqn:Gamma0}.}.
\begin{fact}[Step 2]
	\label{fact:step21}
	There exists a vertex partition
	\begin{align*}
		V(H)
		= W_0 \dcup \bigdcup_{i\in [t]}\left(W_{i,1}\dcup \cdots\dcup W_{i,t'}\right)
		= W_0 \dcup W_1 \dcup \cdots \dcup W_{t_1}
	\end{align*}
	satisfying
	\begin{enumerate}[label=\alabel]
		\item The family $\{W_1,\dots, W_{t_1}\}$ is finer, i.e., $W_{i,j}\subseteq V_i$ for every $i\in [t]$ and $j\in [t']$.
		\item For every $i\in [t]$, we have $V_i \subseteq W_0 \dcup \bigdcup_{j\in [t']}W_{i,j}$.
		\item\label{it:step23} The exceptional class satisfies $V_0\subseteq W_0$ and $|W_0\sm V_0| \le t_1$.
		      Choosing~$N_0$ large enough that $t_1\le \delta 2^{-\l}n$, we consequently have $|W_0|\le (\tau+\delta 2^{-\l})n$.
		\item We have $t_1 =tt'$ and $t_1 < \SZRL(\eps', h, t)$.
	\end{enumerate}
\end{fact}
Let
\[
	|W_1|
	= \cdots
	= |W_{t_1}|
	= m'\,.
\]
We now define the $(\tau + \delta 2^{-\l},t_1,\l')$-partition~$\ccR$ by considering the graphs induced by~$Q^{ij}_\alpha\in \ccQ$ on $K[W_i, W_j]$. We give a more precise description. Given a pair of integers $i,j\in [t_1]$, there exist $i_1,j_1\in [t]$ such that
$W_i\subseteq V_{i_1}$ and $W_j \subseteq V_{j_1}$; we call the classes
$W_1,\dots,W_{t_1}$ \emph{fine}, the classes $V_1,\dots,V_t$ \emph{old}, and
write $\pi(i)=i_1$ for the \emph{old parent} of a fine class~$W_i$. For $i_1\neq j_1$, we consider the partition of the pairs $K[W_i,W_j]$ induced by the bipartite graphs in~$\ccQ$. Let
\[
	R_{\alpha}^{ij} = Q^{i_1j_1}_\alpha\cap K[W_i,W_j] \text{ for all $\alpha\in [\l']$,} \qand K[W_i,W_j] = \bigdcup_{\alpha\in [\l']}R_\alpha^{ij}\,.
\]
On the other hand, if for $i,j\in [t_1]$ we have $i_1=j_1$, then we set $R^{ij}_\alpha =\vn$ for all but one $\alpha \in [\l']$\footnote{As we will see later, the parts $K[W_i,W_j]$ where $i,j\in [t_1]$ and $i_1=j_1$ will never be considered for the final index increment and so we are not careful with how we choose these partitions.}. Let the collection of these bipartite graphs be denoted by
\[
	\ccR = \big\{R^{ij}_\alpha\colon \alpha\in [\l'], ij\in [t_1]^{(2)}\big\}\,.
\]
In view of Theorem~\ref{thm:graphreg}, among the pairs $ij\in[t_1]^{(2)}$
whose two classes have distinct old parents, all but at most $\eps't_1^2$
satisfy
\begin{itemize}[label=($\star$)]
	\item every graph in $\{R^{ij}_\alpha\colon \alpha\in [\l']\}$ is~$\eps'$-regular, and the vertex sets satisfy $W_i\subseteq V_{i_1}$ and $W_j \subseteq V_{j_1}$ for some $i_1j_1\in [t]^{(2)}$.
\end{itemize}
In the final partition~$\ccP_1$, we will only consider the contribution of the~$\ee$-triads that are between parts $W_i,W_j,W_k$ where both $ij,jk\in[t_1]^{(2)}$ satisfy~$(\star)$ above. To this end, we will need the following notation. Let
\[
	\reg_2(\ccR) = \big\{ij\in [t_1]^{(2)}\colon (\star) \text{ holds}\big\}\,.
\]
Every~$R^{ij}_\alpha$ is a subgraph of some $Q^{i_1j_1}_\alpha\in \ccQ$, so~$\ccR$ is finer than~$\ccQ$. Since~$W_0$ may be larger than~$V_0$, however, $\ccR$ need not refine~$\ccQ$. In other words, $\bigcup \ccR\subseteq \bigcup \ccQ$, but it is not necessarily true that $\bigcup \ccR =\bigcup \ccQ$.
\begin{fact}[Step 3]
	\label{fact:step3}
	There exists a $(\tau + \delta 2^{-\l},t_1,\l_1,\eps(\l_1))$-equitable partition~$\ccP_1$ with
	\[
		\ccP_1=\big\{S^{ij}_\alpha\colon ij\in [t_1]^{(2)}\tand \alpha\in [\l_1]\big\}
	\]
	and associated vertex partition
	$V(H) = W_0 \dcup W_1 \dcup W_2\dcup \cdots \dcup W_{t_1}$
	such that the following holds:
	\begin{enumerate}[label=\alabel]
		\item\label{it:step31} For each $ij\in [t_1]^{(2)}$
		      \[
			      K[W_i, W_j] =\bigdcup_{\alpha\in [\l_1]} S_{\alpha}^{ij}
		      \]
		      where the~$S^{ij}_\alpha$ are $(\eps(\l_1), 1/\l_1)$-regular.

		\item\label{it:step32} For each $ij\in \reg_2(\ccR)$ let
		      \[
			      I^{ij}_1 = \{\alpha\in [\l_1]\colon S^{ij}_\alpha \text{ is not a subgraph of any } R^{ij}_\beta\}\,.
		      \]
		      We have
		      \[
			      \sum_{\alpha\in I^{ij}_1} |S^{ij}_\alpha| \le \frac{(m')^2}{\l'}\,.
		      \]
	\end{enumerate}
\end{fact}

\begin{proof}[Proof of Lemma~\ref{lem:indinc} assuming Facts~\ref{fact:step1}, \ref{fact:step21}, and~\ref{fact:step3}]
	We call a pair~$xy$ with $x\in W_i$ and
	$y\in W_j$ \emph{good} if $\pi(i)\ne\pi(j)$, $ij\in\reg_2(\ccR)$, and the unique
	$\alpha\in[\l_1]$ for which $xy\in S^{ij}_\alpha$ does not belong
	to~$I^{ij}_1$.
	Moreover, we refer to $\sum_{P\in\cF}|\ccKee(P)|$ as the \emph{path mass}
	of a family~$\cF$ of ordered~$\ee$-triads.

	Let~$H_{\ccQ}$ be the subhypergraph of~$H$ consisting of the hyperedges whose three
	vertices lie in three pairwise distinct classes among $V_1,\dots,V_t$.
	Thus~$H_{\ccQ}$ consists precisely of the hyperedges supported by the~$\ee$-triads of~$\ccQ$. Define~$H'$ to be the subhypergraph of~$H_{\ccQ}$ consisting of
	those hyperedges $xyz$ for which there are $i,j,k\in[t_1]$ such that
	$x\in W_i$, $y\in W_j$, $z\in W_k$, the old parents
	$\pi(i),\pi(j),\pi(k)$ are pairwise distinct, and all three pairs
	$xy,xz,yz$ are good.
	\begin{clm}
		\label{clm:HminusHprime}
		We have
		\[
			e(H_{\ccQ})- e(H') \le \frac{5n^{5/2}}{\sqrt{\l'}}\,.
		\]
	\end{clm}
	\begin{proof}
		Increase~$N_0$, if necessary, so that Proposition~\ref{prop:crowded}
		applies to~$\ccP_1$ and $t_1n^2\le n^{5/2}/\sqrt{\l'}$.
		Set $C=\sqrt{\l'}$, and let~$\cE_{\rm cr}$ be the set of
		hyperedges of~$H$ lying in at least one crowded~$\ee$-triad
		of~$\ccP_1$. By Proposition~\ref{prop:crowded},
		\[
			|\cE_{\rm cr}|
			\le \sum\bigl\{e_{\ee}(H\,|\,S)\colon S\in\Triad(\ccP_1)\text{ is crowded}\bigr\}
			\le \frac{2}{C}pn^3
			= \frac{2n^{5/2}}C\,.
		\]
		At most $t_1n^2$ hyperedges of~$H_{\ccQ}$ contain a vertex of
		$W_0\sm V_0$.

		Consider a remaining hyperedge in $E(H_{\ccQ})\sm E(H')$ that is
		not in~$\cE_{\rm cr}$. Its vertices lie in fine classes
		with pairwise distinct old parents, and one of its three pairs~$xy$ fails
		to be good. In fact, this is because either the corresponding class-pair~$ij$ belongs to
		\[
			\cB
			=
			\big\{ij\in[t_1]^{(2)}\colon \pi(i)\ne\pi(j),\ ij\notin\reg_2(\ccR)\big\}\,,
		\]
		or because $xy$ belongs to a graph~$S^{ij}_\alpha$ with $\alpha\in I^{ij}_1$.

		In the first case we use that $|\cB|\le \eps't_1^2$ and that, for every
		$ij\in\cB$, the ordered fine~$\ee$-triads whose first pair of classes
		is~$(W_i, W_j)$ have path mass at most $(m')^2n$. Since every hyperedge
		under consideration lies only in~$\ee$-triads of density below~$Cp$, the
		first case accounts for at most $Cp\eps'n^3$ hyperedges.

		In the second case, Fact~\ref{fact:step3}~\ref{it:step32} gives,
		for every $ij\in\reg_2(\ccR)$,
		\[
			n\sum_{\alpha\in I^{ij}_1}|S^{ij}_\alpha|
			\le \frac{n(m')^2}{\l'}\,.
		\]
		Summing over all fine class-pairs, the fine~$\ee$-triads whose first pair
		lies in a graph~$S^{ij}_\alpha$ with $\alpha\in I^{ij}_1$ have path mass at
		most $n^3/\l'$, so the second case accounts for at most $Cpn^3/\l'$
		hyperedges.
		Consequently,
		\[
			e(H_{\ccQ})-e(H')
			\le \frac{2n^{5/2}}C+t_1n^2
			+Cp\left(\eps'n^3+\frac{n^3}{\l'}\right)\,.
		\]
		By the choice of~$N_0$, $t_1n^2\le n^{5/2}/\sqrt{\l'}$. Moreover,
		$\eps'=(\eps(\l_1))^6\le1/\l'$. With $p=n^{-1/2}$ and
		$C=\sqrt{\l'}$, the four terms in the last display are bounded, 
		respectively, by $2$, $1$, $1$, and $1$ times $n^{5/2}/\sqrt{\l'}$.
	\end{proof}
	For each $Q\in\Triad(\ccQ)$, let $\cS(Q)$ consist of the ordered
	$\ee$-triads $S\in\Triad(\ccP_1)$ whose two constituent pair graphs are
	subgraphs of the two constituent pair graphs of~$Q$. These path sets are
	pairwise disjoint subsets of $\ccKee(Q)$, and every hyperedge of~$H'$ supported by~$Q$ is supported by a unique member of~$\cS(Q)$.
	Adjoin to them one further cell consisting of the remaining paths of~$Q$,
	which lie in no member of~$\cS(Q)$, and assign to this cell the $H'$-density zero. Jensen's
	inequality applied to this partition of $\ccKee(Q)$ gives
	\[
		\sum_{S\in\cS(Q)}
		\phi\left(\frac{\dee(H'\,|\,S)}p\right)|\ccKee(S)|
		\ge
		\phi\left(\frac{\dee(H'\,|\,Q)}p\right)|\ccKee(Q)|\,.
	\]
	Summing over the old~$\ee$-triads and using the nonnegativity of all remaining
	fine~$\ee$-triad contributions proves
	\begin{align}
		\label{eqn:indincA3}
		\Ind(H';\ccP_1)\ge \Ind(H';\ccQ)\,.
	\end{align}
	Next, we show that
	\begin{align}
		\label{eqn:indincA4}
		\Ind(H';\ccQ)\ge \Ind(H;\ccQ) - \frac{4D(e(H_{\ccQ})-e(H'))}{pn^3}\,.
	\end{align}
	To this end, let $H''=H_{\ccQ}\sm H'$. For every
	$Q\in\Triad(\ccQ)$ we have $\Eee(H\,|\,Q)=E(H_{\ccQ}\,|\,Q)$, and~$H_{\ccQ}$ is the edge-disjoint union of~$H'$ and~$H''$. Since~$\phi$
	is nondecreasing and $4D$-Lipschitz,
	for every $Q\in\Triad(\ccQ)$ we have $\phi(\dee(H\,|\,Q)/p)-\phi(\dee(H'\,|\,Q)/p)\le 4D\,\dee(H''\,|\,Q)/p$.
	Every hyperedge of~$H''$ occurs in exactly six ordered~$\ee$-triads of~$\ccQ$. Thus,
	using the factor~$1/6$ in the definition of the index, we obtain
	\[
		\Ind(H';\ccQ)
		\ge \Ind(H;\ccQ)
		-\frac{4D}{6pn^3}\sum_{Q\in\Triad(\ccQ)}e_{\ee}(H''\,|\,Q)
		= \Ind(H;\ccQ)-\frac{4D(e(H_{\ccQ})-e(H'))}{pn^3}\,.
	\]
	In view of Claim~\ref{clm:HminusHprime} and~\eqref{eqn:indincA4}, we then have
	\[
		\Ind(H';\ccQ) \ge \Ind(H;\ccQ) - \frac{20D}{\sqrt{\l'}}\,.
	\]
	Together with Fact~\ref{fact:step1} and~\eqref{eqn:indincA3}, this implies
	\[
		\Ind(H';\ccP_1)
		\ge \Ind(H;\ccP) + \frac{\delta^4}{48} - \frac{20D}{\sqrt{\l'}}
		=\Ind(H;\ccP) + \frac{\delta^4}{48} - \frac{20D}{\sqrt{\l2^{2t\l}}}\,.
	\]
	Since $H'\subseteq H$, we have $\Ind(H;\ccP_1)\ge \Ind(H';\ccP_1)$, and this completes the proof of the lemma.
\end{proof}

\begin{proof}[Proof sketch of Fact~\ref{fact:step1}]
	The proof follows the lines of the index increment for graphs in~\cite{scott}. We first focus on one irregular~$\ee$-triad that has bounded density.
	\begin{clm}
		\label{clm:indinc1}
		Let $P=(P^{ij}_\alpha,P^{jk}_\beta)\in \Triad(\ccP)$ be such that~$H$ is not $(\delta,p)$-regular with respect to~$P$ and $\dee(H\,|\,P)\le \delta Dp$.
		Write $P^{ij}=P^{ij}_\alpha$ and $P^{jk}=P^{jk}_\beta$.
		Then there exists a partition
		\[
			P^{ij}=\Theta^{ij}_{0}\cup \Theta^{ij}_{1}, \qquad P^{jk}=\Theta^{jk}_{0}\cup \Theta^{jk}_{1}\,,
		\]
		such that, writing
		$\Theta^{ab}=\left(\Theta^{ij}_{a},\Theta^{jk}_{b}\right)$
		for $(a,b)\in\{0,1\}^2$, we have
		\begin{align}
			\label{eqn:clmstep11}
			\sum_{(a,b)\in\{0,1\}^{2}}
			\phi\left(\frac{\dee\left(H\,|\,\Theta^{ab}\right)}{p}
			\right)|\ccKee(\Theta^{ab})|
			\ge
			\left(\phi\left(\frac{\dee(H\,|\,P)}{p}\right)+\delta^{3}\right)|\ccKee(P)|\,.
		\end{align}
	\end{clm}
	\begin{proof}
		Since~$H$ is not $(\delta,p)$-regular with respect to~$P$, there is a~$\ee$-partite subgraph $Q=(A_0,B_0)\subseteq P=(A,B)$ such that
		\[
			|\ccKee(Q)|\ge \delta|\ccKee(P)|
			\qand
			|\dee(H\,|\,Q)-\dee(H\,|\,P)|>\delta p\,.
		\]
		Set $A_1=A\sm A_0$, $B_1=B\sm B_0$, and
		$\Theta^{ab}=(A_a,B_b)$ for $a,b\in\{0,1\}$. Put
		\[
			K=|\ccKee(P)|,\qquad
			w_{ab}=\frac{|\ccKee(\Theta^{ab})|}{K},\qquad
			x_{ab}=\frac{\dee(H\,|\,\Theta^{ab})}{p},\qquad
			x=\frac{\dee(H\,|\,P)}p\,.
		\]
		Here $K>0$ by~\eqref{eqn:triadpathlowerbound}.
		The four path sets $\ccKee(\Theta^{ab})$ form a disjoint partition
		of $\ccKee(P)$. Consequently,
		\[
			\sum_{a,b}w_{ab}=1,
			\qquad
			\sum_{a,b}w_{ab}x_{ab}=x,
			\qquad
			w_{00}\ge\delta,
			\qquad
			|x_{00}-x|>\delta\,.
		\]
		Since $x\le\delta D\le D/4<2D$, we have $\phi(x)=x^2$. For
		$y\ge0$ define the remainder above the supporting tangent to~$\phi$ at~$x$ by
		\[
			\rho_x(y)=\phi(y)-\phi(x)-2x(y-x)\,.
		\]
		Directly from the definition of~$\phi$,
		\[
			\rho_x(y)=
			\begin{cases}
				(y-x)^2,                & 0\le y\le2D, \\
				(2D-x)^2+(4D-2x)(y-2D), & y\ge2D.
			\end{cases}
		\]
		Thus $\rho_x(y)\ge0$ for every $y\ge0$. In addition,
		we have $\rho_x(x_{00})>\delta^2$. This is immediate if $x_{00}\le2D$, while if
		$x_{00}\ge2D$, then
		\[
			\rho_x(x_{00})
			\ge(2D-x)^2
			\ge((2-\delta)D)^2
			>\delta^2\,.
		\]
		Using the averaging identities above, we now obtain
		\begin{align*}
			\frac1K\sum_{a,b}
			\phi(x_{ab})|\ccKee(\Theta^{ab})|-\phi(x)
			&=\sum_{a,b}w_{ab}\rho_x(x_{ab})
			+2x\sum_{a,b}w_{ab}(x_{ab}-x)              \\
			&=\sum_{a,b}w_{ab}\rho_x(x_{ab})
			\ge w_{00}\rho_x(x_{00})
			>\delta^3\,.
		\end{align*}
		This proves~\eqref{eqn:clmstep11}.
	\end{proof}

	\begin{clm}
		\label{clm:indinc2}
		There are at least $\delta\l^2t^3/2$ $\ee$-triads $P\in \Triad(\ccP)$ such that~$H$ fails to be $(\delta,p)$-regular with respect to~$P$ and $\dee(H\,|\,P) \le \delta Dp$.
	\end{clm}
	\begin{proof}
		Let~$q$ be the number of~$\ee$-triads $P\in\Triad(\ccP)$ with
		$\dee(H\,|\,P)>\delta Dp$. By~\eqref{eqn:triadpathlowerbound} and the
		sixfold counting of hyperedges by ordered~$\ee$-triads,
		\[
			q\,\delta Dp\frac{n^3}{4\l^2t^3}
			<\sum_{P\in\Triad(\ccP)}e_{\ee}(H\,|\,P)
			\le6e(H)
			\le6pn^3\,.
		\]
		Consequently,
		\[
			q
			<\frac{24}{\delta D}\l^2t^3
			=\frac\delta2\l^2t^3\,,
		\]
		where the last equality uses $D=48/\delta^2$. Since at least
		$\delta\l^2t^3$ $\ee$-triads are irregular, the claim follows.
	\end{proof}
	We simultaneously partition every bipartite graph in $\ccP= \{P^{ij}_\alpha\colon ij\in [t]^{(2)}, \alpha \in [\l]\}$ using the witnesses of irregularity in the following way.

	Given $P^{ij}_\alpha\in \ccP$, there are at most~${2t\l}$ choices of $k, \beta$ such that $(P^{ij}_\alpha,P^{ik}_\beta)$ or $(P^{ij}_\alpha, P^{jk}_\beta)$ forms an irregular~$\ee$-triad, where we use
	that $P^{ij}_\alpha$ and~$P^{ji}_\alpha$ denote the same bipartite graph. The reverse orientations~$P^{ijk}_{\alpha\beta}$ and~$P^{kji}_{\beta\alpha}$ have the same associated graph, and we use the same irregularity witness for both, so they do not create additional partitions. For each of these irregular~$\ee$-triads, using Claim~\ref{clm:indinc1} above, we obtain a (possibly trivial) partition of~$P^{ij}_\alpha$ into two parts $\Theta^{ij}_0\dcup \Theta^{ij}_1$. We then take the common refinement of these partitions over all irregular~$\ee$-triads.

	As a consequence, we obtain a partition of $K[V_i,V_j]$ into $\l' = \l\cdot 2^{2t\l}$ parts, some of which may be empty. Denote this partition by
	\[
		K[V_i,V_j] = \bigdcup_{\alpha = 1}^{\l'} Q^{ij}_{\alpha}\,.
	\]
	The collection of pairwise disjoint bipartite graphs thus obtained will be denoted by
	\[
		\ccQ= \{Q^{ij}_\alpha\colon ij\in [t]^{(2)}, \alpha\in [\l']\}\,.
	\]
	For every irregular~$\ee$-triad of density at most $\delta Dp$,
	Claim~\ref{clm:indinc1} gives a raw increase of more than
	$\delta^3|\ccKee(P)|$. Subsequent common refinements do not decrease this
	contribution, by convexity. Combining Claim~\ref{clm:indinc2},
	\eqref{eqn:triadpathlowerbound}, and the factor~$1/6$ in the index, we obtain
	\begin{align*}
		\Ind(H;\ccQ)-\Ind(H;\ccP)
		 & \ge \frac1{6n^3}\cdot\frac\delta2\l^2t^3\cdot
		\delta^3\cdot\frac{n^3}{4\l^2t^3}
		=\frac{\delta^4}{48}\,,
	\end{align*}
	which concludes the discussion of the proof of Fact~\ref{fact:step1}
\end{proof}
\begin{proof}[Proof sketch of Fact~\ref{fact:step3}]
	In Step~2 we obtained the partition~$\ccR$, which has associated vertex partition
	\[
		V(H) = W_0 \dcup W_1\dcup W_2\dcup \cdots \dcup W_{t_1}\,,
	\]
	and a collection of bipartite graphs
	\[
		\ccR= \{R^{ij}_\alpha\colon \alpha \in [\l'], ij\in [t_1]^{(2)}\}\,,
	\]
	such that for every $ij\in \reg_2(\ccR)$, we have the partition
	\begin{align*}
		K[W_i, W_j] = \bigdcup_{\alpha\in [\l']}R^{ij}_\alpha\,,
	\end{align*}
	where the~$R^{ij}_\alpha$ are~$\eps'$-regular.

	Our goal is to obtain the $(\tau+ \delta 2^{-\l},t_1,\l_1,\eps(\l_1))$-equitable partition~$\ccP_1$ satisfying Fact~\ref{fact:step3}~\ref{it:step31} and~\ref{it:step32}.
	Note that for $ij\notin \reg_2(\ccR)$, Fact~\ref{fact:step3} does not claim any relation between the partition of $K[W_i,W_j]$ by graphs~$S^{ij}_\beta$ and that by graphs~$R^{ij}_\alpha$, and so we simply choose an arbitrary partition into $(\eps(\l_1),1/\l_1)$-regular graphs (such a partition can, for instance, be chosen randomly). Consequently, we only focus on partitioning $K[W_i,W_j]$ where $ij\in \reg_2(\ccR)$.

	To partition into graphs of equal densities, we will use the slicing lemma below. We omit the straightforward proof of the following fact (based on random partitioning). See also~\cites{FR02, RS04} for more details.
	\begin{fact}[Slicing lemma]
		\label{fact:slicing}
		For every $\eps>0$ and $\g\l\in \NN$, there exists $m_0(\eps,\g\l)$ such that the following holds for every $\g m\geq m_0(\eps,\g\l)$. If~$F$ is an~$\eps$-regular bipartite graph with
		\[
			V(F) = V_1(F)\dcup V_2(F) \tand |V_1(F)|= |V_2(F)| = \g m\,,
		\]
		then there exist an integer $r(F)$ and a partition
		\[
			F= \bigdcup\bigl\{F(\gamma)\colon \gamma\in \{0,\dots,r(F)\}\bigr\}\,,
		\]
		satisfying
		\begin{enumerate}[label=\alabel]
			\item\label{it:SlicingSparse2} The graphs $F(\gamma)$ are $(2\eps,1/\g\l)$-regular for every $1\le \gamma\le r(F)$.
			\item\label{it:SlicingSparse1} The exceptional part satisfies $|F(0)| < (\g m)^2/\g\l$.
		\end{enumerate}
	\end{fact}
	We fix $ij\in\reg_2(\ccR)$ and apply Fact~\ref{fact:slicing} to each~$R^{ij}_\alpha$ for $\alpha\in [\l']$ with $\eps = \eps'$ and $\g\l = \l_1$. As a consequence, we obtain the following:
	\[
		R^{ij}_\alpha = R^{ij}_\alpha(0)\dcup\bigdcup \{R^{ij}_\alpha(\gamma)\colon 1\le \gamma \le r(R^{ij}_\alpha)\}\,,
	\]
	where each of the graphs $\{R^{ij}_\alpha(\gamma)\colon \gamma \ge 1\}$ is $(2\eps', 1/\l_1)$-regular, and further
	\[
		|R^{ij}_\alpha(0)|\le (m')^2/\l_1\,.
	\]
	Let
	\[
		\Gamma^{ij}_0 = \bigdcup_{\alpha\in [\l']}R^{ij}_\alpha(0)
	\]
	and observe that since $\l_1 = 4(\l')^2$, we have
	\[
		|\Gamma_0^{ij}|
		\le \l' \cdot (m')^2/\l_1
		= \frac{(m')^2}{4\l'}\,.
	\]
	We now have the new partition of $K[W_i, W_j]$ given by
	\[
		K[W_i, W_j] = \Gamma_0^{ij}\dcup\bigdcup_{\alpha=1}^{\l'}\bigdcup_{\gamma=1}^{r(R^{ij}_\alpha)}R^{ij}_\alpha(\gamma)\,.
	\]
	To obtain a `perfect partition', we further slice~$\Gamma^{ij}_0$; similar
	ideas were used in~\cites{NRS06,RS07}.
	If $\Gamma^{ij}_0$ is empty, then $r(\Gamma^{ij}_0)=0$ and no positive-indexed~$\Gamma^{ij}_0$-piece arises below.

	Each positive-indexed piece has at least
	$(1/\l_1-2\eps')(m')^2$ edges, whence 
	\[
		r(R^{ij}_\alpha)
		\le 
		\frac{\l_1}{1-2\eps'\l_1}\,.
	\]
	Since $2\eps'\l_1(\l_1+1)<1$, this integer is at most~$\l_1$ for every~$\alpha$,
	so there are at most $\l'\l_1$ such pieces. Since each is $2\eps'$-regular, their
	complement~$\Gamma_0^{ij}$ is~$\g\eps$-regular, where
	\[
		\g\eps = 2\eps'\l_1\l'\,.
	\]
	Put $\eps_1=\eps(\l_1)$ and $\eps_2=2\g\eps$. Since
	$\eps'=\eps_1^6$, $\eps_1\le(8\l_1^2)^{-1}$, and
	$\l_1=4(\l')^2$, we have
	\[
		2\eps'\le\eps_2,
		\qquad
		(\l_1+1)\eps_2\le\eps_1,
		\qquad
		\eps_2<\frac1{\l_1(\l_1+1)}\,.
	\]
	We then apply Fact~\ref{fact:slicing} to~$\Gamma^{ij}_0$, obtaining the partition
	\begin{align}
		\label{eqn:Gamma0}
		\Gamma_0^{ij} =\Gamma_0^{ij}(0)\dcup\bigdcup_{\beta = 1}^{r(\Gamma_0^{ij})}\Gamma^{ij}_0(\beta)\,,
	\end{align}
	where the graphs in $\{\Gamma^{ij}_0(\beta)\colon \beta\ge 1\}$ are all $(2\g\eps,1/\l_1)$-regular and
	\[
		|\Gamma_0^{ij}(0)|\le (m')^2/\l_1\,.
	\]
	As a consequence, $K[W_i,W_j]$ is partitioned as follows:
	\begin{align*}
		K[W_i,W_j] = \Gamma_0^{ij}(0)\dcup\bigdcup_{\beta = 1}^{r(\Gamma_0^{ij})}\Gamma^{ij}_0(\beta)\dcup \bigdcup_{\alpha=1}^{\l'}\bigdcup_{\gamma=1}^{r(R^{ij}_\alpha)}R^{ij}_\alpha(\gamma)\,.
	\end{align*}
	Further, since $|\Gamma^{ij}_0|\le (m')^2/(4\l')$, in view of~\eqref{eqn:Gamma0} we have
	\begin{align}
		\label{eqn:factstep3b}
		\sum_{\beta = 0}^{r(\Gamma_0^{ij})}|\Gamma^{ij}_0(\beta)|
		\le \frac{(m')^2}{4\l'}\,.
	\end{align}
	Write $Z=\Gamma_0^{ij}(0)$. Then $K[W_i,W_j]$ is the disjoint union of~$Z$ and
	\[
		r = r(\Gamma^{ij}_0) + \sum_{\alpha\in[\l']}r(R^{ij}_\alpha)
	\]
	graphs that are all $(\eps_2,1/\l_1)$-regular. Each of these graphs has
	between $(1/\l_1-\eps_2)(m')^2$ and
	$(1/\l_1+\eps_2)(m')^2$ edges, while $|Z|\le(m')^2/\l_1$.
	If $r\ge\l_1+1$, their lower size bound exceeds~$(m')^2$, and if
	$r\le\l_1-2$, their upper size bound together with~$Z$ is smaller
	than~$(m')^2$. The inequalities displayed above for~$\eps_2$ therefore give
	$r\in\{\l_1-1,\l_1\}$.

	If $r=\l_1-1$, the complement~$Z$ of these~$r$ graphs is
	$(r\eps_2,1/\l_1)$-regular and hence
	$(\eps_1,1/\l_1)$-regular. If $r=\l_1$, then the lower size bound
	for the~$r$ graphs gives $|Z|\le\l_1\eps_2(m')^2$.
	
	If $r(\Gamma_0^{ij})\ge1$, absorb~$Z$ into
	$\Gamma_0^{ij}(1)$. If $r(\Gamma_0^{ij})=0$, choose any one of the
	positive-indexed graphs $R^{ij}_\alpha(\gamma)$ and absorb~$Z$ into that
	graph instead. In either case, the resulting union is
	$((\l_1+1)\eps_2,1/\l_1)$-regular and therefore
	$(\eps_1,1/\l_1)$-regular.

	In this way, in both cases we obtain a partition of $K[W_i,W_j]$ into~$\l_1$ graphs 
	that are all $(\eps_1,1/\l_1)$-regular. Denote this partition 
	by $K[W_i,W_j] = \bigdcup_{\beta\in [\l_1]}S^{ij}_\beta$.
	
	Let~$\cJ^{ij}$ consist of those final parts that are not equal to a single
	graph $R^{ij}_\alpha(\gamma)$. Every~$S^{ij}_\beta$ with $\beta\in I^{ij}_1$ belongs to~$\cJ^{ij}$.
	If $r=\l_1-1$, or if $r=\l_1$ and
	$r(\Gamma_0^{ij})\ge1$, all graphs in~$\cJ^{ij}$ together are
	contained in~$\Gamma_0^{ij}$, and hence their total number of edges is at most
	$(m')^2/(4\l')$ by~\eqref{eqn:factstep3b}. In the remaining case,
	$r(\Gamma_0^{ij})=0$ and the only graph placed in~$\cJ^{ij}$ is the union of~$Z$ with one graph
	$R^{ij}_\alpha(\gamma)$. Its size is at most
	\begin{align*}
		\left(\frac1{\l_1}+(\l_1+1)\eps_2\right)(m')^2
		 & \le \left(\frac1{\l_1}+\eps_1\right)(m')^2
		\le \frac{2(m')^2}{\l_1}
		\le \frac{(m')^2}{\l'}\,.
	\end{align*}
	Thus in every case
	\[
		\sum_{\beta\in I^{ij}_1}|S^{ij}_\beta|
		\le\sum_{S\in\cJ^{ij}}|S|
		\le\frac{(m')^2}{\l'}\,,
	\]
	which proves Fact~\ref{fact:step3}~\ref{it:step32}.
\end{proof}

\end{document}